\documentclass[11pt]{amsart}
\usepackage[utf8]{inputenc}
\usepackage{graphicx}
\usepackage{xstring}
\usepackage{forloop}
\usepackage{bbm, amsmath, amsfonts, amscd, latexsym, amsthm, amssymb, graphicx, mathrsfs, comment}
\usepackage[colorlinks=true]{hyperref}

\usepackage{abstract}
\usepackage{mathtools}
\usepackage{caption}
\usepackage{MnSymbol}
\usepackage[dvipsnames]{xcolor}

\usepackage{tikz-cd}
\tikzset{
  symbol/.style={
    draw=none,
    every to/.append style={
      edge node={node [sloped, allow upside down, auto=false]{$#1$}}}
  }
}

\makeatletter
\newif\if@check@engine  \@check@enginetrue 
\makeatother
\usepackage[usenames,dvipsnames]{pstricks}

\usepackage{booktabs}

\newtheorem{theor}{\hspace{1cm}{\sc Theorem}}[section]
\newtheorem{utver}[theor]{\hspace{1cm}{\sc Proposition}}
\newtheorem{sledst}[theor]{\hspace{1cm}{\sc Corollary}}
\newtheorem{lemma}[theor]{\hspace{1cm}{\sc Lemma}}

\newtheorem{assum}[theor]{\hspace{1cm}{\sc Assumption}}
\newtheorem*{utver*}{\hspace{1cm}{\sc Proposition}}
\theoremstyle{definition}

\newtheorem{defin}[theor]{\hspace{1cm}{\sc Definition}}
\newtheorem{exa}[theor]{\hspace{1cm}{\sc Example}}

\newtheorem{rem}[theor]{\hspace{1cm}{\sc Remark}}

\newcommand{\Vol}{\mathop{\rm Vol}\nolimits}

\newcommand{\codim}{\mathop{\rm codim}\nolimits}

\newcommand{\sing}{\mathop{\rm sing}\nolimits}

\newcommand{\conv}{\mathop{\rm conv}\nolimits}

\newcommand{\MV}{\mathop{\rm MV}\nolimits}

\newcounter{idx}

\newcommand{\rotraise}[1]{
  \StrLen{#1}[\slen]
  \forloop[-1]{idx}{\slen}{\value{idx}>0}{
    \StrChar{#1}{\value{idx}}[\crtLetter]
    \IfSubStr{tlQWERTZUIOPLKJHGFDSAYXCVBNM}{\crtLetter}
      {\raisebox{\depth}{\rotatebox{180}{\crtLetter}}}
      {\raisebox{1ex}{\rotatebox{180}{\crtLetter}}}}
}

\renewcommand{\emph}[1]{{\it {\color{NavyBlue} #1}}}

\def\R{\mathbb R}

\def\Z{\mathbb Z}
\def\Q{\mathbb Q}
\def\C{\mathbb C}
\def\CC{({\mathbb C}^\star)}

\def\CP{\mathbb C\mathbb P}

\begin{document}

\begin{center}{\Large \sc Engineered complete intersections: slightly degenerate Bernstein--Kouchnirenko-Khovanskii}

\vspace{3ex}

{\sc Alexander Esterov} 
\end{center}

\begin{abstract}
Geometry of sparse systems of polynomial equations (i.e. the ones with prescribed monomials and generic coefficients) is well studied in terms of their Newton polytopes. The results of this study are colloquially known as the Bernstein--Kouchnirenko--Khovanskii toolkit, and unfortunately are not applicable to many important systems, whose coefficients slightly fail to be generic.

This for instance happens if some of the equations are obtained from another one by taking partial derivatives or permuting the variables, or the equations are linear, realizing a non-trivial matroid, or in more advanced settings such as generalized Calabi--Yau complete intersections. 

Such interesting examples (as well as many others) turn out to belong to a natural class of ``systems of equations that are nondegenerate upon cancellations''. 
We extend to this class several classical and folklore results of the Bernstein--Kouchnirenko--Khovanskii toolkit, such as the ones regarding the number and regularity of solutions, their irreducibility, tropicalization and Calabi--Yau-ness.

\end{abstract}

\tableofcontents

\section{Introduction}\label{sconcrintro}
\subsection{The BKK toolkit} Upon choosing a coordinate system $(x_1,\ldots,x_n):T\xrightarrow{\sim}\CC^n$ in a complex torus $T$, its characters $m:T\to\CC$ can be written as monomials $m(x)=x^a:=x_1^{a_1}\cdots x_n^{a_n},\,a\in\Z^n$, and thus form an integer lattice $M\simeq\Z^n$. Linear combinations of these characters are called {\it Laurent polynomials} and regarded as functions on $T$. For a finite set of monomials $A\subset M$, their linear combinations are said to be {\it supported} at $A$ and form a vector space $\C^A$.
Given $k$ support sets $A_i\subset M$, a collection of Laurent polynomials $$f=(f_1,\ldots,f_k)\in\C^{\mathcal A}:=\C^{A_1}\oplus\cdots\oplus\C^{A_k}\eqno{(*)}$$defines an algebraic set $\{f=0\}
\subset T$.

Discrete characteristics of geometric nature (such as the number of components or the Euler characteristics) take the same value on the set $\{f=0\}$ for almost all $f\in\C^{\mathcal A}$, and this typical value has an expression (of equally geometric nature) in terms of the support ${\mathcal A}$. 
Kouchnirenko, Berstein, Khovanskii, Varchenko and others found such expressions for several important characteristics of $\{f=0\}$. 
Recall some of these theorems, and the genericity assumption they impose on $f$. 
\begin{defin}\label{defnondeg000}
1. For a linear function $l\in M^*$, the {\it $l$-degree} of a monomial $x^a$ is $l(a)$.

2. The {\it $l$-leading part} $f^l_i$ of a polynomial $f_i$ is
the sum of its terms of the maximal $l$-degree.

3. A system of equations $f=0$ is nondegenerate, if $0$ is a regular value of the map $f^l:=(f_1^l,\ldots,f_k^l)$ for every $l\in M^*$ (or, equivalently, $f^l=0$ is a {\it regular system of equations}).

4. {\it Almost all} $f\in\C^{\mathcal{A}}$ (i.e. all $f$ outside a suitable hypersurface in $\C^{\mathcal{A}}$) are nondegenerate. The smallest such exceptional hypersurface is well defined (\cite{adv}) and called the {\it discriminant}. 
\end{defin}
\begin{exa}\label{exabkk00}
1. For $l(a)=a_1+\cdots+a_n$, we get the usual notion of the degree and the leading part. 
For almost all polynomials of prescribed degrees $f_i$, the tuple $f$ is nondegenerate, and such $f=0$ defines a regular subvariety in $\CP^n$, transversal to every coordinate plane.

2. Assume $f=(f_1,\ldots,f_k)$ is nondegenerate, and $A_i\subset M$ is the set of monomials participating in $f_i$ (called the {\it support set}).
Then, for $k=n$, the number of roots of $f=0$ equals the mixed volume (Definition \ref{defpolyh}) of the convex hulls of $A_1,\ldots,A_n$ (called the {\it Newton polytopes}): $$|\{f=0\}|=A_1\cdots A_n:=\MV(A_1,\ldots,A_n).$$For $k=2$, the manifold $\{f_1=f_2=0\}$ has Euler characteristics $$e(A_1,A_2):=(-1)^{n}\sum\nolimits_{k=1}^{n-1}A_1^kA_2^{n-k}
,\eqno{(**)}$$and is connected unless $A_1$ and $A_2$ belong to parallel 2-planes, or one of them to a line. 

To review of this BKK toolkit, see the original papers \cite{bernst}, \cite{kh75} and \cite{kh15} respectively, or Section \ref{sbkk}. In particular, see Theorem \ref{bkk1} and Example \ref{exared} for the Euler characteristics and connectedness of $\{f=0\}$ respectively if $k>2$.
\end{exa}

\subsection{A BKK toolkit for one slightly degenerate system of equations}  \label{s12}

If the system $f_1=f_2=0$ is degenerate (again assuming $k=2$ for simplicity), we might still hope that:

(1) $f^l_1=0$ defines a regular hypersurface $H\subset T$;
and

(2) the bad leading part $f^l_2$ can be cancelled in a suitable difference $\tilde f_2:=f_2-gf_1$ so that $\tilde f^l_2=0$, unlike its predecessor $f^l_2=0$, is a regular equation on $H$.
\begin{exa}\label{exa000}
Let $f_1\in\C^{A_1}$ be generic, and $f_2$ be $\partial f_1/\partial x_1$. Then the system $f_1=f_2=0$ 
is degenerate: once all points of $A$, at which $l:\Z^n\to\Z$ takes its maximal value, have the first coordinate $\lambda$, we have $f_2^l=\lambda f_1^l/x_1$. However $\tilde f_2:=f_2-\lambda f_1/x_1$ justifies the above hope.
\end{exa}

This observation leads to a BKK toolkit for the set $\{f=\partial f/\partial x_1=0\}$ defined by generic $f\in\C^A,\,A\subset\Z^n$.
To express the geometry of this set in terms of $A$, let $H_b\subset\Z^n$ be the hyperplane of points whose first coordinate equals $b$. 
Define the {\it incremental polytope} $\hat A$ of this complete intersection as the intersection of the convex hulls of $(A+A)\setminus H_{b}$ over all $b\in\Z$.

\begin{theor}[proved in Section \ref{scancelcrit}]\label{th0cci}
1. The set $\hat A$ is a lattice polytope. 

2. Almost all polynomials $f\in\C^A$ satisfy the condition of Example \ref{exa000} and define a smooth complete intersection $S_1=\{f=\partial f/\partial x_1=0\}$.
All such manifolds $S_1$ are diffeomorphic.

3. Their Euler characteristics equals (see remark below for how to evaluate these formulas):
$$
\frac{A}{1-A}\cdot \frac{\hat A-A}{1-\hat A+A}= e(\hat A-A,A)=e(A,A)-\sum\nolimits_{b\in\Z}\left(e(A,A)-e(A,A\setminus H_b)\right).$$

4. Their tropical fan equals $[\hat A]\cdot[A]-[A]\cdot[A]$ (see remark below for how to evaluate).

5. They are irreducible (equivalently, connected), if $A$ can be shifted to the interior of $\hat A$.

6. Their closure in the $\hat A$-toric variety are Calabi--Yau, if $\hat A$ is reflexive.
\end{theor}
\begin{rem}\label{remcritci}
1. To evaluate the first formula for the Euler characteristics, expand $1/(1-X)$ as $1+X+X^2+\cdots$, open the brackets, and evaluate every monomial of the form $A_1\cdot\cdots\cdot A_n$ as the mixed volue of the convex hulls of $A_1,\ldots,A_n$ (ignoring monomials of degrees different from $n$). For the definition of $e(\cdot,\cdot)$ in the further formulas, see $(**)$ in Example \ref{exabkk00}. 

2. To evaluate the tropical fan, recall that $[X]$ denotes the dual tropical fan to the convex hull of $X$ (i.e. the corner locus of the support function of $X$), and then evaluate the expression in the ring of tropical fans, see e.g. \cite{ms}. 

3. The irreducibility condition (Part 5) is not a criterion. It would be interesting to classify $A$ such that $S_1$ is reducible (or at least to understand to what extent this classification is finite).
Indeed, there do exist support sets $A$ for which the curve $S_1$ is not connected.

For example, if $n=3$ and $A\subset H_b\cup H_{b'}$, then $S_1$ consists of several lines parallel to the first coordinate axis. To see this, rewrite the defining system of equations $f=f_{x_1}=0$ as $(f-bx_1f_{x_1})/x_1^b=(f-b'x_1f_{x_1})/x_1^{b'}=0$, which do not depend on $x_1$.

Similarly, if $A\subset H_b\cup(\Z\cdot a)$ for $a\in\Z$, then $S_1$ has several components in cosets of the form $\{x^a=const\}\subset\CC^3$ (cf Example \ref{exared} and Theorem \ref{bkk3} for the classical BKK toolkit).

4. When it comes to computing the Euler characteristics of $S_1$ for particular support sets, the parenthesized differences in the second formula for the euler characteristics can be simplified: for instance, for $n=3$, first, $e(A,A)-e(A,A\setminus H_b)=\MV(A,A,A)-\MV(A,A\setminus H_b,A\setminus H_b)=\sum_E \MV(A,A,A)-\MV(A,A\setminus E,A\setminus E)$ over all edges $E\subset H_b$ of the convex hull $\conv A$.

Second, denoting by $\pi_E$ a surjection of lattices $\Z^3\to\Z^2$ sending the edge $E$ to a point $e\in\Z^2$, the difference $\MV(A,A,A)-\MV(A,A\setminus E,A\setminus E)$ equals the lattice length of $E$ times the lattice area of the polygon $(\conv\pi_EA)\setminus\conv(\pi_EA\setminus e)$, see Proposition \ref{answerlink3}.

\end{rem}

\subsection{Nondegeneracy upon cancellations.} 
We come back to the hope in the beginning of Section \ref{s12}. While in general this hope might look both unmotivated and greedy, it is actually close to $f=0$ being a sch\"on complete intersection (SCI) in the sense of \cite{schon} (even equivalent to it once $H$ is connected, see Theorem \ref{thnewtcinondegexplic}). Moreover, it covers all of our motivating examples of SCIs (such as Example \ref{exa00}). 
We now formulate it for an arbitrary number of equations $k$.
\begin{defin}\label{defnuc}
1. Laurent polynomials $f=(f_1,\ldots,f_k)$ on the torus $T$ are said to be {\it cancellable}, if, for every non-zero $l\in M^*$, we have an upper unitriangular $k\times k$ matrix of Laurent polynomials $C_l$ such that the {\it $l$-cancelled system of equations} $f_l:=(f\cdot C_l)^l=0$ defines a {\it complete intersection} (i.e. $\codim\{f_{l,1}=\cdots=f_{l,i}=0\}=i$ for every $i\leqslant k$). 

2. The tuple $f$ is said to be {\it nondegenerate upon cancellations} $C_l$ (abbreviated to NUC), if $f_{l,1}=\cdots=f_{l,i}=0$ is a regular system of equations for every $l\ne 0$ and  $i\leqslant k$.

\end{defin}

We shall describe the Euler characteristics, connectedness, tropicalization and Calabi--Yau-ness of arbitrary NUC complete intersections (Proposition \ref{tropcancel} and Theorem \ref{thcatchup1}, or more restrictive and elementary Theorem \ref{thcatchup}). This general answer will be made explicit for some of the following particularly interesting examples, defined by a generic polynomial $f\in \C^A$. 

\begin{exa}\label{exa00}
{\bf (1) Engineered complete intersections.} Let $B_1,\ldots,B_q\subset\Z^n$ be finite sets, 
and $f_1,\ldots,f_k$ be given linear combinations of generic Laurent polynomials $g_i\in\C^{B_i}$ (where the genericity assumption depends on the given linear combinations). Then the system $f_1=\cdots=f_k=0$ is NUC (Corollary \ref{corolccinondeg}). The case of one-point support sets $B_i$ is the most general, but considering larger supports is convenient for engineering complete intersections with prescribed geometry, see for instance Example 1.3.6 
in \cite{schon}. 

This should not be confused with another generalization of the BKK toolkit: we take the same linear combinations $f_j$ of polynomials $g_i$, but this time the latters are arbitrarily fixed, while the formers are generic. Then the Euler characteristics of $f_1=\cdots=f_k=0$ equals the signed volume of the respective Newton--Okounkov body \cite{kk}. If now $g_i$'s are especially nice (formalized in terms of SAGBI/canonical/Khovanskii's bases), then the NO body is a lattice polytope, making the answer combinatorially tractable. This does not overlap with our story and was elaborated in \cite{bmmt}, with an initial motivation to study the harmonic balance equations for Duffing oscillators (Theorem 3.1 therein; by a funny coincidence, that system is at the same time engineered too).

See Section \ref{scancelcrit} and especially Proposition \ref{remmainengin} for the summary of what we know about the geometry of engineered complete intersections (including formulas for the Euler characteristics and tropical fan, sufficient conditions for irreducibility and Calabi--Yau-ness). These general results specialize, in particular, to the following interesting special cases.

{\bf (2) Critical complete intersections.} Expanding on Example \ref{exa000}, let $g\in\C^{\mathcal A}$ be a generic tuple of Laurent polynomials. 
Then any system of equations formed by partial derivatives (or antiderivatives, of any orders, or leading parts) of $g_i$'s is engineered. Examples include:

-- the critical locus of the projection of the hypersuarface $\{f=0\}$ to the first coordinate hyperplane, given by the equations $f=\partial f/\partial x_1=0$ (studied in Theorem \ref{th0cci});

-- the $k$-th Thom-Bordmann stratum of this projection: $f=\partial f/\partial x_1=\cdots=\partial^k f/\partial x_1^k=0$; 

-- the critical points of a coordinate function restricted to $\{f=0\}$, given by the equations $f=\partial f/\partial x_2=\cdots=\partial f/\partial x_n=0$ (studied in Theorem \ref{thconcr2} below);

-- the critical points of $f$: $\partial f/\partial x_1=\partial f/\partial x_2=\cdots=\partial f/\partial x_n=0$ (studied in Theorem \ref{thconcr2});

\noindent Unlike in the classical BKK, the convex hull of $A$ does not determine topology of these sets.

{\bf (3) Symmetric complete intersections.} The complete intersection $f(x_1,x_2,x_3,\ldots,x_n)=f(x_2,x_1,x_3,\ldots,x_n)=0$ 
splits into components that are NUC. We study such symmetric varieties here (Theorem \ref{th0sci}, similar to \ref{th0cci}) and then in a separate paper \cite{symm}, motivated by their applications to Galois theory and using the techniques from the present paper.

{\bf (4) Hyperplane arrangement complements.} They are defined by systems of linear equations in the torus. Every such system (not only generic one) is engineered, hence NUC. 

{\bf (5) Generalized Calabi--Yau complete intersections.} They are defined by special systems of polynomial equations in products of projective spaces \cite{gcicy} and are actively studied in physics (see \cite{bh} for an overview). These systems of equations are degenerate, i.e. do not belong to CICYs that are classically studied as in \cite{bb94}. 

{\bf (6) Qualitative polynomial optimization.} Given polynomials $f_0,\ldots,f_k$, the real extrema of $f_0$ on $f_1=\cdots=f_k=0$ are studied by polynomial optimization, and the number of complex extrema is known as the {\it algebraic degree}. The most studied cases include:

-- {\bf the maximum likelihood (ML) degree} \cite{chks}: $f_0$ is a generic monomial (of great geometric importance regardless of the optimization perspective, see e.g. \cite{huh});

-- {\bf the polar degree:} $f_0$ is generic linear (of great geometric importance regardless of the optimization perspective, see e.g. \cite{dp} or, more recently, \cite{sst});

-- {\bf the euclidean distance (ED) degree:} $f_0$ is the distance to a generic point \cite{dhost}.

The algebraic degree is the number of critical points of the Lagrange multipliers function $F(x,\lambda)=f_0(x)+\lambda_1f_1(x)+\cdots+\lambda_kf_k(x)$. If $f_i$'s are generic polynomials supported at $A_i\subset \Z^n$, then $F$ is a generic polynomial supported at $A_0*\cdots *A_k:=\bigcup_ie_i \times A_i\subset\Z^k\times\Z^n$, where $e_0,\ldots,e_k$ are the vertices of the standard simplex in $\Z^k$. Thus, the algebraic degree is expressed in terms of $A_i$'s by Theorem \ref{thconcr2} below. This extends to the setting when $f_1=\cdots=f_k=0$ is NUC.

\end{exa}

For a face $\Gamma$ of the convex hull of $A\subset\Z^n$, let $(-1)^{n-\dim\Gamma}e^\Gamma_A$ be the Euler obstruction of the $A$-toric variety at its $\Gamma$-orbit (see Proposition \ref{eulobstr} for its combinatorial expression). 
By $\Vol \Gamma$ we denote the lattice volume of $\Gamma$: this is the volume form on the affine span ${\rm aff}\, \Gamma$ normalized so that the minimal volume of a lattice simplex in ${\rm aff}\, \Gamma$ equals 1.
\begin{theor}[proved in Section \ref{proofth3}] \label{thconcr2}
For finite $A\subset\Z^n$ and generic $f\in\C^A$ (as in Remark \ref{remgenerdf}), the strata $S^1=\{f=\partial f/\partial x_2=\cdots=\partial f/\partial x_n=0\}$ and $\{df=0\}$ consist of $\sum_{\Gamma}e_A^\Gamma\Vol \Gamma$ points, where $\Gamma$ ranges over the following collections of faces of $\conv A$ (including $\conv A$ itself):

-- faces parallel to the first coordinate axis, for $S^1$;

-- faces whose affine span contains $0\in\Z^n$, for $\{df=0\}$ \newline (in the latter case we additionally assume with no loss of generality that $0\notin A$).
\end{theor}
\begin{rem}
In the absence of ``bad'' faces $\Gamma$ (e.g. if 0 is in the interior of the convex hull of $A$), the answer is just the lattice volume of $A$. If $A$ is a Delzant polytope (i.e. the $A$-toric variety is smooth), then all Euler obstructions equal $1$. Even in this generality, applying this to the Lagrangian multipliers function of Example \ref{exa00}.5, we arrive at interesting applications (c.f. \cite[Theorem 1]{sottile} and \cite[Corollary 5.11]{dhost}).
\end{rem}

\subsection{BKK for NUC complete intersections} Given a cancellable complete intersection $f_1=\cdots=f_k=0$ in the torus $T\simeq\CC^n$ with the character lattice $M\simeq\Z^n$, its geometry is obviously not determined by the Newton polytopes of its equations. We first explain what other combinatorial data we need, and then, in terms of this additional data, express some geometry of the complete intersection (specializing to the above Theorem \ref{th0cci}, among others).

\begin{defin}\label{defincr}
A lattice polytope $A_j$ in $M$ is called the $j$-th {\it incremental polytope} of $f=0$, if, for every $l\in M^*$ such that the cancelled complete intersection (Definition \ref{defnuc}) $\{f_{l,1}=\cdots=f_{l,j-1}=0\}$ is not empty, the $l$-degree $\deg_l f_{l,1}+\cdots+\deg_l f_{l,j}$ equals $\max l(A_j)$.
\end{defin}
For instance, in Example 1.3.6 
in \cite{schon}, the polytope on the right of the picture is the incremental polytope of the engineered NUC complete intersection.

\begin{theor}[proved in Section \ref{scancel}]\label{thcatchup} 
1. If a complete intersection $S:=\{f_1=\cdots=f_k=0\}$ has incremental polytopes $A_j$, then its tropical fan equals $[A_1]\cdot[A_2- A_1]\cdot\ldots\cdot[A_k-A_{k-1}].$

2. If moreover $S$ is NUC, then  it has isolated singularities, and its Euler characteristics is
$$\prod_{i=1}^k\frac{A_i-A_{i-1}}{1-A_i+A_{i-1}}+(-1)^{n-k}(\mbox{the sum of the Milnor numbers of the singularities, if any}).$$

3. If, moreover, each $A_i$ can be shifted into the interior of $A_{i+1}$, and $k<n$, then $S$ is connected (and, moreover, its first $n-k$ Betti numbers equal those of $T$).

4. In particular, in this case $S$ is irreducible, unless it is a singular curve.

5. The closure of $S$ in the $A_k$-toric variety is Calabi--Yau, if $S$ is smooth and $A_k$ is reflexive and compatible with the fan of part (1) in the following sense.
\end{theor}
\begin{defin}\label{defunicorn}
A polytope $P$ is said to be {\it compatible} with an $m$-dimensional fan $F$, if $F$ is contained in the $m$-skeleton of the dual fan $[P]$ (and thus equals a union of its cones). 
\end{defin}
This follows from Theorem \ref{thcatchup1}, which answers the same questions for arbitrary NUC complete intersections (not assuming the existence of incremental polytopes). 

\begin{rem}\label{remlikely}
1. Furthermore, if we evaluate in the ring of tropical fans the expression $$\prod_{i=1}^k\frac{[A_i]-[A_{i-1}]}{1-[A_i]+[A_{i-1}]},$$ then, for a smooth NUC $S\subset T$, we get the full tropical characteristic class of $S$ (in the sense of \cite{e13}), with its lowest and highest components recovering Parts 1 and 2 of the theorem. 

In particular, descending (as in \cite{fs}) from tropical fans to cohomology of a toric variety $X_\Sigma\supset T$ for a sufficiently fine fan $\Sigma$, the tropical characteristic class $\langle S\rangle$ evaluates to the CSM class of the closure of $S$ in $X_\Sigma$. (A sufficient fineness condition on $\Sigma$ is that the tropical fan  representing each component of the characteristic class $\langle S\rangle$ is a union of cones of $\Sigma$.)

2. The proof of Theorems \ref{thcatchup1} and, in turn, \ref{thcatchup}, refers to the Morse theoretic technique of \cite{schon}. It would be important to find a technique that works over arbitrary fields.
\end{rem}

\subsection{Structure of the paper}

The classical BKK toolkit is summarized in Section \ref{sbkk}. Regularity of complete intersections (Sections \ref{fdiscr}\&\ref{sstransv}) and mixed volumes (Section \ref{slocmv}) are given more attention, in order to cover what we later need for NUC complete intersections.

Section \ref{sincr} is a detailed study of two lattice polytopes that later turn out to be the incremental polytopes of two simplest cancellable complete intersections: $f=\partial f/\partial x_1=0$ and $f(x_1,x_2,x_3,\ldots)=f(x_2,x_1,x_3,\ldots)=0$.

In Section \ref{scci}, we study engineered complete intersections. Along the way, we digress to study NUC complete intersections in general: we describe their tropicalization in Proposition \ref{tropcancel}, and then, in its terms, other geometric properties of a NUC complete intersection (Theorem \ref{thcatchup1}). Both of these steps specialize to engineered complete intersections (Proposition \ref{remmainengin}), as soon as we prove that they are NUC (Theorem \ref{thccinondegexplic}). The same is then done for another class of NUC complete intersections, the symmetric ones (which are not engineered), see Section \ref{ssci}.

All of these results on NUC complete intersections are derived from \cite{schon}, because being NUC implies being sch\"on in the sense of \cite{schon}. This key fact is proved in Section \ref{scancel}.

The formula for the number of critical points of a generic polynomial $f\in\C^A$ (Theorem \ref{thconcr2}, which may be familiar to experts) is given a detailed proof in the appendix (Section \ref{proofth3}).

Not to get lost: 
we study four classes of complete intersections in $\CC^n$: sch\"on, Newtonian, nondegenerate upon cancellaions (NUC), and engineered. They are related as follows:

\vspace{1ex}

\begin{center}

\includegraphics[scale=0.45]{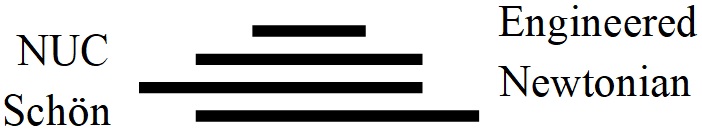}

\end{center}

\section{The BKK toolkit}\label{sbkk}

\subsection{Homogeneity}
When we say that something about polynomials or sets in an algebraic torus is valid by homogeneity considerations, we mean the following.
\begin{defin}\label{dhomog}
1) We say that a collection of subsets $V_i\subset T$ and Laurent polynomials $g_j:T\to\C$ is {\it $k$-homogeneous}, if one can choose coordinates $T\simeq\CC^n$ so that the subsets $V_i\subset\CC^n$ and hypersurfaces $\{g_j=0\}\subset\CC^n$ are preimages of certain $\tilde V_i\subset\CC^k$ and $\{\tilde g_j=0\}\subset\CC^k$ under the coordinate projection $\CC^n\to\CC^k$. The latters are called the {\it dehomogenization} of the initial collection.

2) A statement is valid {\it by homogeneity} for a collection of $V_i\subset T$ and $g_j:T\to\C$, if 

-- this collection is $k$-homogeneous; 

-- the sought statement reduces to the same for the dehomogenization $(\tilde V_i, \tilde g_j)$;

-- the statement for the dehomogenization is known.
\end{defin}
\begin{exa}
A Laurent polynomial is $k$-homogeneous, if all of its monomials belong to the same $k$-dimensional affine subspace of the character lattice.

\end{exa}

\subsection{Regularity} Denote by $\sing (f)$ the singular locus $\{x\,|\, f(x)=0,\, df(x)=0\}$ of a function $f$.
\begin{lemma}\label{fdiscr}
Let $f_i$ be finitely many regular functions on a smooth algebraic variety $M$, then the singular locus of their generic linear combination is contained in $(\cap_i\sing (f_i))\cup($the positive dimensional part of $\cap_i\{f_i=0\})$.
\end{lemma}
\begin{rem}
1. One cannot omit the positive dimensional part of $\cap_i\{f_i=0\}$ in this statement: consider e.g. $f_1(x,y)=xy$ and $f_2(x,y)=x$ where $\sing (f_2)=\emptyset$ and $\cap_i\{f_i=0\})=\{x=0\}$. Observe in this case that the inclusion $\sing (af_1+bf_2)\subset \{x=0\}$ is proper.

2. In the space of all linear combinations of $f_i$'s, the set of linear combinations with a larger singular locus is known as the $\{f_\bullet\}$-discriminant (and especially $A$-discriminant, if $A\subset\Z^n$ is a finite set of monomial functions on $M=\CC^n$).

3. Applying the lemma to a collection $A\subset\Z^n$ of monomials on a torus $\CC^n$, we conclude that a generic Laurent polynomial $f\in\C^A$ defines a regular hypersurface $f=0$ in $\CC^n$.
\end{rem}
{\it Proof.} A linear combination $\sum_{i=0}^N c_i f_i$ fails to satisfy the sought condition, if it has a singular point 

(i) at one of the finitely many isolated points $y_j$ of the zero locus $f_0=\cdots=f_N=0$ outside of $\cap_i\sing f_i$, or

(ii) in the complement $M_0\subset M$ to this zero locus.

The tuples $$c=(c_1:\cdots:c_N)$$ producing a singular point at $y_j$ form a projective subspace $H_j\subset\CP^N$ given by the system of linear equations $\sum_i c_idf_i(y_j)=0$ (which is non-trivial since $y_j\notin\cap_i\sing f_i$).

To characterize the tuples $c$ producing a singular point in $M_0$, consider the projection $\pi$ of the smooth variety $$\{(c,y)\,|\,\sum_ic_if_i(y)=0\}\subset\CP^N\times M_0$$ to the first multiplier $\CP^N$. By the Thom transversality lemma, if $c$ is a regular value of $\pi$, then the zero locus of the linear combination $\sum_ic_if_i(y)$ is smooth in $M_0$. By the Bertini--Sard theorem, the singular values of $\pi$ form a set, whose closure $D$ is proper algebraic.

Thus, for all tuples $c$ outside the proper algebraic subset $D\cup\bigcup_j H_j\subset\CP^N$, the linear combination $\sum_ic_if_i(y)$ satisfies the sought condition.
\hfill$\square$

\subsection{Regular complete intersections}\label{sstransv}
\begin{defin}\label{defdegroot}
Let $g_i$ be smooth functions (or, more generally, sections of line bundles) on a smooth manifold $M$.
We say that $x\in M$ is a degenerate root of the system of equations $g_1=\cdots=g_m=0$, if $g_i(x)=0$, and the differentials $dg_i(x)$ are linearly dependent.

The system of equations is said to be regular on $M$ if none of its roots (if any) is degenerate.
\end{defin}
\begin{rem}\label{remreg0}
1) Regularity implies that the roots of the system form a smooth codimension $m$ manifold (but not vice versa).

2) if $m>\dim M$, then every root is degenerate, so regularity means there are no roots.
\end{rem}

We mostly need the following special case. Let $V$ and $W$ be algebraic subsets in the torus $T\simeq\CC^n$ such that $V$ is smooth outside of $W$, and the dimension of $W$ is at most $m$ (at every its point). 
\begin{defin}
The system of $k$ equations $f=0$ is said to be regular on $(V,W)$, if it defines a set of dimension at most $m-k$ in $W$, and is regular on $V\setminus W$. 
\end{defin}
Given finite sets ${\mathcal A}=(A_1,\ldots,A_k)$ in the character lattice $M$ of the torus $T$, the following is well known.
\begin{theor} \label{bkktransv}
Being regular on $(V,W)$ is a generic property for $f\in\C^{\mathcal A}$. 
\end{theor}
We shall prove below a Proposition \ref{propregind}, which is stronger and applies to engineered complete intersections as well.
\begin{rem}[\cite{khovcomp}]\label{remreg}
1. In particular, for $V=T$, the theorem implies that $f=0$ is a smooth codimension $k$ submanifold of the torus $T$.

2. Note that, if $k>\dim V$ (or just $k>n$ in the case $V=T$), then $f=0$ does not intersect $V$ at all (or is just empty in the case $V=T$) by Remark \ref{remreg0}.2.

3. Slightly more generally, the regularity of $f=0$ in the torus $T$ implies its emptiness, if all $A_i$'s can be shifted in $M$ to the same sublattice of dimension smaller than $k$. This follows from (2) by homogeneity.

\end{rem}

\begin{defin}
In the setting of Theorem \ref{bkktransv}, a vector subspace $L\subset \C^{\mathcal A}=\C^{A_1}\oplus\cdots\oplus\C^{A_k}$ is said to be regular for $(V,W)$, if generic $f\in L$ is regular on $(V,W)$.
\end{defin}
\begin{utver}\label{propregind}
If the projection $L'\subset \C^{A_1}\oplus\cdots\oplus\C^{A_{k-1}}$ of a vector subspace $L\subset\C^{\mathcal A}$ is regular for $(V,W)$, and $(0,\ldots,0,x^a)\in L$ for some $a\in A_k$, then $L$ is regular for $(V,W)$.
\end{utver}

{\it Proof of Proposition \ref{propregind}.}
Our assumption implies that there exists an algebraic hypersurface $D\subset L'$ such that the system of equations $f_1=\cdots=f_{k-1}=0$ is regular on $(V,W)$ once $(f_1,\ldots,f_{k-1})\in L'\setminus D$.

Decomposing $f_k$ into $\hat f_k-cx^a,\hat f_k\in\C^{A_k\setminus\{a\}}$, it is enough to check that, for given $(f_1,\ldots,f_{k-1})\in L'\setminus D$ and any $\hat f_k$, 
the system of equations $f=0$ is regular on $(V,W)$ for all but finitely many $c\in\C$. We now show that yes: such exceptional $c$ are finitely many.

Regularity of $f=0$ on $W$: the set $W'=W\cap\{f_1=\cdots=f_{k-1}=0\}$ has codimension $k-1$ in $W$, so the zero locus of $f$ has codimension $k$ in $W$ unless $\hat f_k/x^a=c$ identically on some irreducible component of $W'$, i.e. unless $c$ is one of the finitely many constants which the function $\hat f_k/x^a$ may identically equal on one of the finitley many components of $W'$.

Regularity of $f=0$ on $V$: the system of equations $f_1=\cdots=f_{k-1}=0$ is regular on $V\setminus W$, thus defining a smooth submanifold $V'$ in it; $f=0$ is regular on $V\setminus W$ once $f_k=0$ is regular on $V'$, i.e. unless $c$ is one of the finitely many critical values of the function $\hat f_k/x^a$ on $V'$.\hfill$\square$

\vspace{1ex}

{\it Proof of Theorem \ref{bkktransv}.}  This theorem states that the space $L=\C^{\mathcal A}$ is regular on $(V,W)$. Since such $L$ satisfies the assumption of Proposition \ref{propregind} by induction on $k$, Theorem \ref{bkktransv} follows.\hfill$\square$

\subsection{Mixed volume} \begin{defin}\label{defpolyh} The {\it lattice mixed volume} $\MV$ is the (unique) real-valued function of $n$ convex sets in $\R^n$, which is 
i) symmetric in its arguments, 

ii) linear in each argument w.r.t the Minkowski summation $B_1+B_2:=\{b_1+b_2\,|\,b_i\in B_i\}$,

iii) has the diagonal value $\MV(B,\ldots,B)$ equal to the lattice volume $\Vol_\Z (B):=n!\Vol (B)$. \end{defin}
\begin{exa}\label{examv2dim} The mixed volume is the polarization of the volume regarded as a degree $n$ polynomial function on the space of polytopes. For instance, if $n=2$, we have
$$\MV(B_1,B_2)=\Vol(B_1+B_2)-\Vol(B_1)-\Vol(B_2).$$
\end{exa}
\begin{rem}\label{remmv}
1. The multiplier $n!$ assures that the mixed volume of lattice polytopes (i.e. polytopes whose vertices belong to $\Z^n$) is integer. When we refer to mixed volumes of subsets of $\Z^n$, we imply mixed volumes of their convex hulls in $\R^n$.

2. More generally, for any integer lattice $M\simeq\Z^n$ and $n$ convex sets in its ambient vector space $\R\otimes M$, their $M$-lattice mixed volume is defined as the mixed volume of their images under an induced isomorphism $\R\otimes M\xrightarrow{\sim}\R\otimes\Z^n=\R^n$ (the value does not depend on the choice of the identification).

3. More generally yet, if $k$ convex sets can be shifted to the same $k$-dimensional rational subspace $K\subset \R\otimes M$, their $k$-dimensional lattice mixed volume can be defined as the $K\cap M$-lattice mixed volume of their shifted copies in $K$. The value does not depend on the choice of the shifts (and is 0 independently of $K$, if $K$ is not unique, i.e. if the sets can be shifted to a less than $k$ dimensional plane). 
\end{rem}
We refer to \cite{ew} for a detailed introduction to mixed volumes in the presence of lattice.

\subsection{Geometry of nondegenerate complete intersections}

Given a Laurent polynomial $f$ on a torus $T\simeq\CC^n$, recall that its support set in the character lattice $M\simeq\Z^n$ is the set of monomials, participating in $f$ with non-zero coefficients, and its Newton polytope in $M\otimes\R\simeq\R^n$ is the convex hull of the support set.

If a system of equations $f_1=\cdots=f_k=0$ is nondegenerate (in the sense of Definition \ref{defnondeg000}), then the topology of its zero locus $\{f_1=\cdots=f_k=0\}$ depends only on its Newton polytopes (see e.g. Corollary 2.16 
in \cite{schon}). 
We recall how simplest topological characteristics of this set express in terms of the Newton polytopes.

\begin{theor}[\cite{kh75}] \label{bkk1}

Assume $f=(f_1,\ldots,f_k)$ is nondegenerate, and $A_i\subset M$ is the support set of $f_i$.
The Euler characteristics of the complete intersection $\{f=0\}$ 
equals $$e_n(A_1,\ldots,A_k):=(-1)^{n-k}\sum_{d_1+\ldots+d_k=n\atop d_1,\ldots,d_k>0} A_1^{d_1}\cdots A_k^{d_k},$$ 
where the product of $n$ finite subsets denotes the lattice mixed volume of their convex hulls.

\end{theor}

This fact can be seen as a special case of Theorem \ref{thcatchup} (see Example 1.10.1 
in \cite{schon} for details), and generalizes the classical expression for the Euler characteristics of a smooth projective complete intersection. However, in contrast to projective complete intersections, even for a generic $f\in \C^{\mathcal A}$ and $k<n$, the set $\{f=0\}$ may consist of more than one component. 
\begin{exa}\label{exared0}
If the first $q$ of the sets $A_1,\ldots,A_k$ are contained in the first $q$-dimensional coordinate subspace of $\Z^n$, then the equations $f_1,\ldots,f_q$ depend only on the first $q$ of the standard coordinates $x_1,\ldots,x_n$, and in general position they have exactly $A_1\cdots A_q$ common roots $(x_1^j,\ldots,x_q^j)$ in $\CC^q$ (Example \ref{exabkk00}). 
As a consequence, the set $\{f=0\}$ splits into $A_1\cdots A_q$ subsets, each contained in a torus coset given by the equations $(x_1,\ldots,x_q)=(x_1^j,\ldots,x_q^j)$.

Each torus coset is isomorphic to $\CC^{n-q}$, and the corresponding part of $\{f=0\}$ is defined in $\CC^{n-q}$ by the equations $$f_i(x_1^j,\ldots,x_q^j,x_{q+1},\ldots,x_n)=0,\, i=q+1,\ldots,k. \eqno{(**_j)}$$The support sets of these equations $B_i\subset\Z^{n-q}$ are the projections of $A_i$'s along the first coordinate plane $\Z^q\subset\Z^n$.

Note that the BKK toolkit can be used to study the zero locus of the equations $(**_j)$: the multivalued map $F:\C^{\mathcal A}\to\C^{\mathcal B}$, sending $f$ to $(**_j)$ for all $j$, is dominant, so the general position of  the initial tuple $f$ in $\C^{\mathcal A}$ ensures the general position of the tuple $(**_j)$ in $\C^B:=\bigoplus_i\C^{B_i}$  (in particular nondegeneracy).

\end{exa}

We now repeat this example avoiding the use of coordinates. For this, we shall refer to a {\it saturated} sublattice $L\subset M\simeq\Z^n$ 
(i.e. an intersection of $M$ with a vector subspace in $M\otimes\Q\simeq\Q^n$)
as a {\it subspace} of $M$, its coset of the form $L+m=\{l+m\,|\, l\in L\},\,m\in M$, as an {\it affine subspace}, a connected algebraic subgroup $G$ of the tours $T\simeq \CC^n$ as a {\it subtorus}, and its coset of the form $G\cdot x=\{g\cdot x\,|\,g\in G\},\, x\in T$, as a {\it shifted subtorus}.

\begin{exa}\label{exared}
If the first $q$ of the set $A_1,\ldots,A_q$ are contained in affine subspaces of the lattice $M$ parallel to the same $q$-dimensional subspace $L\subset M$, then the set $\{f=0\}$ for generic $f\in\C^{\mathcal A}$ splits into $A_1\cdots A_q$ subsets (this mixed volume is understood in the sense of Remark \ref{remmv}.3), each subset contained in a shifted copy of a certain subtorus of $T$. (More specifically, this is the subtorus defined by the equations $m(x)=1$ for all monomials $m\in L$.)

For generic $f\in\C^{\mathcal A}$, each subset of the complete intersection $\{f=0\}$ is itself a complete intersection defined in its shifted subtorus by generic equations supported at certain sets $B_j\subset M/L$. (More specifically, the set $B_j$ is the image of $A_j,\, j=q+1,\ldots,k$, under the projection $M\to M/L$.)
This example can be reduced to the preceding one by choosing an isomorphism $T\simeq\CC^n$ so that the corresponding isomorphism $M\simeq\Z^n$ sends $L$ to the first coordinate plane.
\end{exa}
We need the coordinate free form for this example in order to make an important statement: this is actually the only example when a generic complete intersection $\{f=0\}$ is reducible.
Morse specifically, once we choose maximal possible $q$ for given supports $A_1,\ldots,A_k$ in the setting of this example, each of the described pieces of the set $\{f=0\}$ is nonempty and irreducible, so $\{f=0\}$ has exactly $A_1\cdots A_q$ components. This was proved by Khovanskii:
\begin{theor}[\cite{kh15}]\label{bkk3}
A smooth complete intersection defined by nondegenerate equations with support sets $A_1,\ldots,A_k\subset M$ is irreducible (and hence connected), unless $q<n$ of the sets $A_j$ can be shifted to the same $q$-dimensional subspace of $M$, and their lattice mixed volume in this subspace exceeds 1.

\end{theor}
\begin{rem}
1. If none of $A_i$'s belongs to a hyperplane, then this theorem follows from Theorem \ref{thcatchup}, because, in its notation, $N_i\supset A_i$.

2. Tuples of lattice polytopes having mixed volume at most 1 are classified in \cite{eg11}. 
\end{rem}

\begin{rem}\label{remmvprod}
If $k=n$ in the setting of Example \ref{exared0}, then, applying the BKK formula 
to the initial system of equations $f_1=\cdots=f_n=0$, its ``square'' subsystem $f_1=\cdots=f_q=0$, and its ``quotient'' system $(**_j)$, we have the chain of equalities: $\MV(A_1,\ldots,A_n)=($the number of solutions of $f_1=\cdots=f_n=0)=\sum_j($the number of solutions of $(**_j))=($the number of solutions of $f_1=\cdots=f_q=0)\cdot($the number of solutions of $(**_1))=\MV(A_1,\ldots,A_q)\cdot\MV(B_{q+1},\ldots,B_n)$.

Thus, in the setting of Example \ref{exared} with $k=n$, we have the same important formula for the mixed volume: $$\MV(A_1,\ldots,A_n)=\MV(A_1,\ldots,A_q)\cdot\MV(B_{q+1},\ldots,B_n).$$
\end{rem}

\subsection{Localizing mixed volumes}\label{slocmv} Here we collect two formulas for mixed volumes, which we were unable to find in the literature (though they will hardly surprise the experts). 

\begin{utver}\label{locmv}
For convex bodies $A_1,\ldots,A_n,B_1,\ldots,B_n\subset\R^n$, let $\lambda\in\Lambda$ be the connected components of the set of linear functions $l\in(\R^n)^*$ at which the support functions differ:
$$(A_1(l),\ldots,A_n(l))\ne(B_1(l),\ldots,B_n(l)).$$Denoting by $B^\lambda_i$ the polytope whose support function equals
$B_i(l)$ for $l\in \lambda$ and $A_i(l)$ otherwise,
$$A_1\cdots A_n-B_1\cdots B_n=\sum_{\lambda\in\Lambda}A_1\cdots A_n-B^\lambda_1\cdots B^\lambda_n.$$
\end{utver}
Approximating the convex bodies with rational polytopes, the statement reduces the lattice polyhedral case. In this generality, the equality is established by computing the mixed volumes on both sides with the height$\times$base formula (see e.g. Theorem 4.10 in \cite{ew}). The polytopes $B^\lambda_i$ not always exist, but in many important cases they do (see e.g. the next section). 

\vspace{1ex}

For the second identity, let $A\subset M\simeq\Z^n$ be a finite set.

\begin{defin}\label{defpolylink}

1. Refer to $\dim\conv A$ as $\dim A$, 
and to a face of $\conv A$ intersected with $A$ -- as a face of $A$.
For a subset $B$ and a $k$-dimensional face $Q$ in $A$, let $p:M\to\Z^{n-k}$ be an affine surjection sending $B$ to 0. Define the {\it link cone} $[A/Q]:=\R_+\cdot\conv pA\subset\R^{n-k}$ and the (non-convex) {\it link polytope} $[A-B/Q]:=[A/Q]\setminus (pB+[A/Q]).$

When we want to refer to it in terms of a complementary set $R$ such that $B=A\setminus R$, instead of $B$ itself, we write $[A\cap R/Q]$.

\vspace{1ex}

2. Finite sets $A$ and $B\subset M$ are said to be compatible, if their dual fans coincide, or, equivalently, if $A, B$ and $A+B$ have the same number of faces.

If they do, there is a natural 1-to-1 correspondence between their faces, such that the Minkowski sum of the corresponding faces of $A$ and $B$ is the corresponding face of $A+B$. The face of $B$ corresponding to the face $E$ of $A$ is denoted by $B^E$.
\end{defin}

\begin{exa}
Denoting the $z$-coordinate axis by $Z$, the hatched polygons on the picture are $[A_1\cap Z/A_1\cap Z]$ and $[A_2\cap Z/A_2\cap Z]$.
\begin{center}

\includegraphics[scale=0.8]{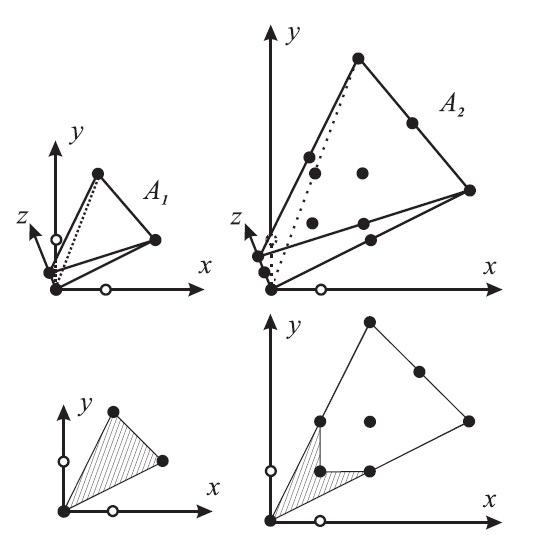}

\end{center}
 \end{exa}

\begin{utver}\label{locvol}
1. For compatible $A_0,\ldots,A_k\subset M$, choose $A\subset A_0$ so that $A_0\setminus A$ contains at most $k$-dimensional faces of $A_0$, and denote the set of such $k$-dimensional faces by $\mathcal{A}$. Then, defining $A_Q:=((\mbox{vertices of }A_0)\setminus Q)\cup A$, we have $$A_0^{n-k}A_1\cdots A_k-A^{n-k}A_1\cdots A_k=\sum_{Q\in\mathcal{A}} A_0^{n-k}A_1\cdots A_k-A_Q^{n-k}A_1\cdots A_k=\sum_{Q\in\mathcal{A}} A_0^{n-k}A_1^Q\cdots A_k^Q-A_Q^{n-k}A_1^Q\cdots A_k^Q.$$ 

2. Each difference on the right hand side
equals $(A_1^Q\cdots A_k^Q) \cdot \Vol_\Z [A_0-A/Q]$.
\end{utver}
Part 1 is established by computing the mixed volumes on both sides with the height$\times$base formula (see e.g. Theorem 4.10 in \cite{ew}). Part 2 is the formula of Remark \ref{remmvprod}, applied to the mixed volumes $A_\bullet^{n-k}A_1^Q\cdots A_k^Q$ on the right hand side of Part 1.

\vspace{1ex}

We now give a common generalization of Propositions \ref{locmv} and \ref{locvol}.1: it explains their similarity, reduces apparent serendipity in defining $B_i^\lambda$ and $A_Q$, allows to avoid a bothersome application of the height$\times$base formula, and covers further interesting examples. It applies to continuous piecewise homogeneous polynomial functions $m$ and $\tilde m$ of degree $d$ on a $d$-dimensional tropical fan $F$ (with $F=\R^d$ and important special case).

\begin{theor}\label{locdelta}
Let $\lambda\in\Lambda$ be the connected components of the set of $l\in F$ at which $m(l)\ne\tilde m(l)$. Then, defining $M_\lambda(l)$ to be $\tilde m(l)-m(l)$ for $l\in\lambda$, and 0 otherwise, we have
$$\tilde m-m=\sum\nolimits_{\lambda\in\Lambda} M_\lambda,\quad \mbox{ and therefore }\quad \delta^d(\tilde m\cdot F)-\delta^d (m\cdot F)=\sum\nolimits_{\lambda\in\Lambda} \delta^d (M_\lambda\cdot F)\in\Z,$$

where $\delta$ is the corner locus operator as defined in \cite{mmjpoly}.
\end{theor}
This ``theorem'' is tautological, and admits an equally tautological corollary, in case $\tilde m=\Phi_d(\tilde m_1,\ldots,\tilde m_p)$ and $m=\Phi_d(m_1,\ldots,m_p)$, where $\tilde m_i$ and $m_i$ are continuous piecewise linear functions, and $\Phi_d$ is the $d$-homogeneous component of a power series $\Phi\in\R[[x_1,\ldots,x_p]]$.
\begin{sledst}\label{locdeltalin}
Let $\lambda\in\Lambda$ be the connected components of the set of $l\in F$ at which $$\bigl(m_1(l),\ldots,m_p(l)\bigr)\ne\bigl(\tilde m_1(l),\ldots,\tilde m_p(l)\bigr).$$ Then, defining the function $\tilde m_{i,\lambda}(l)$ to be $m_i(l)$ for $l\in\lambda$, and $\tilde m_i(l)$ otherwise, we have
$$\delta^d\Phi(\tilde m_1,\ldots,\tilde m_p)F-\delta^d\Phi(m_1,\ldots,m_p)F=\sum\nolimits_{\lambda\in\Lambda} \delta^d\Phi(\tilde m_1,\ldots,\tilde m_p)F-\delta^d\Phi(\tilde m_{1,\lambda},\ldots,\tilde m_{p,\lambda})F\in\R,$$
where $\delta^d\Phi(\bullet,\ldots,\bullet)F$ stands for $\delta^d(\Phi_d(\bullet,\ldots,\bullet)\cdot F)$.
\end{sledst}
\begin{rem}
This applies to mixed volumes,  because the mixed volume of polytopes $P_1,\ldots,P_n$ can be obtained from the product of the support functions of $P_1,\ldots,P_d$ on the dual tropical fan $F=[P_{d+1}]\cdot\ldots\cdot[P_n]$ by applying the corner locus operator $\frac{\delta^d}{d!}$, see \cite{mmjpoly}. \end{rem}
This remark makes Propositions \ref{locmv} and \ref{locvol}.1 special cases of Corollary \ref{locdeltalin}:

\ref{locmv}) $d=n$, $\tilde m_i$ and $m_i$ are the support functions of the polytopes $A_i$ and $B_i$, and $F=\R^n$;

\ref{locvol}) $d=n-k$, $\tilde m_i$ and $m_i$ are the support functions of $A_0$ and $A$, and $F=[A_1]\cdot\ldots\cdot[A_k]$.

\begin{rem} In the second application, the elements of $\Lambda$ are relatively open maximal cones of the fan $F$. Note that if we instead try to take $d=n$, $F=\R^n$, $m_i(l)=\tilde m_i(l)=A_i(l)$ for $i=1,\ldots,k$, and $m_i(l)=A(l)$ and $\tilde m_i(l)=A_0(l)$ for larger $i$, then $\Lambda$ consists of one element, and Corollary \ref{locdeltalin} becomes useless. That is why we cannot restrict it to the case $F=\R^d$.
\end{rem}

\subsection{Singularities of toric varieties}
Assume that $A$ in the character lattice $M\simeq\Z^n$ of the torus $T\simeq\CC^n$ cannot be shifted to a proper sublattice. 
Then the map $T\to\CP^A$ sending $x\in T$ to the point with the homogeneous coordinates $x^a,\,a\in A$, is injective. The closure of its image $X_A$ is a toric variety, extending the action of $T$ on itself. The orbits of this action are in one to one correspondence with the faces $B\subset A$, and will be denoted by $O_B\subset X_A$.

In general, $X_A$ is singular at the points of $O_B$ for $B\ne A$. Since its singularities at every point of $O_B$ are equivalent, we can study them in terms of the combinatorial data $(A,B)$.

\begin{utver}[Section 1.5 in \cite{dcg}]\label{eulobstr}
1. The multiplicity $c^B_A$ of $X_A$ at a point of a codimension $k$ orbit $O_B$ equals the lattice $k$-volume of the (non-convex) lattice polytope $P_B:=\pi(\conv(A)\setminus\conv(A\setminus B))$, where $\pi:M\to\Z^k$ is an affine surjection satisfying $\pi(B)=0$.

2. (cf. \cite{kiyoshi}) The Euler obstruction $(-1)^{\dim A-\dim B}e^B_A$ of $X_A$ at the same point is defined by the matrix equality $(e^{B'}_B)=(c^{B'}_B)^{-1}$, where $c^B_B:=1$ and $c^{B'}_B=0$ unless $B'\subset B$.

\end{utver}

Note that $P_B$ is the link polytope $[A\cap B/B]$, in terms of Definition \ref{defpolylink}.

\section{Incremental polytopes: two instances}\label{sincr}
For a given set $A\subset\Z^n$, we introduce lattice polytopes $\hat A$ and $\check A$. They will happen to be incremental polytopes for the simplest critical and symmetric complete intersections respectively, though we make no appeal to them in this elementary geometric section.

We recall that: $H_b\subset\Z^n$ is the hyperplane of points whose first coordinate equals $b\in\Z$;

-- the {\it support function} of a set $A\subset\R^n$ is a function $A(\bullet):(\R^n)^*\to\R$, whose value $A(\gamma)$ equals the supremum of the linear function $\gamma:\R^n\to\R$ on $A$;

-- the {\it $\gamma$-support face} is the subset $A^\gamma\subset A$ of points at which $\gamma$ equals its supremum $A(\gamma)$.
\subsection{The critical incremental}
\begin{theor}\label{hatunicorn}
Let $\hat A\subset\R^n$ be the intersection of convex hulls $\hat A_{(b)}$ of the sets $(A+A)\setminus H_b$.

1. This is a lattice polytope, and its support function is the minimum of those of $\hat A_{(b)}$.

2. If $\hat A^\gamma$ is an edge, then the polytopes $\hat A^\gamma$ and $\conv A^\gamma$ can be shifted to the same 2-dimensional sublattice of $\Z^n$, in which their mixed area equals the area of $\conv A^\gamma$.
\end{theor}
\begin{rem} \label{remhat} 1. Part 2 essentially means that $\hat A$ is compatible (Definition \ref{defunicorn}) with the fan $[A]([\hat A]-[A])$, see Corollary \ref{corolcritunic}.

2. Clearly, we have $\hat A^\gamma=\widehat{A^\gamma}$.
\end{rem}
The proof makes use of the following toy version of the removable singularity theorem.
\begin{utver}[{\bf Removable disvexity theorem}]\label{removeconcav} Let $S\subset\R^n$ be a closed polyhedral (or even subanalytic) set of codimension exceeding 1.
If a function $g:\R^n\setminus S\to\R$ is convex near every point of its domain, then it extends to a convex function on the whole $\R^n$.  
\end{utver}
We shall apply this to the following $S$ and $g$.
\begin{lemma} \label{lhat}
Let $S$ be the set of $\gamma$ such that $A^\gamma$ is not contained in a line. Then, near every $\gamma\notin S$, the function $g(\bullet):=\min_b \hat A_{(b)}(\bullet)$ equals $\min(\hat A_{(2p)}(\bullet),\hat A_{(2q)}(\bullet))$ for some $p$ and $q$.
\end{lemma}
For this lemma, we look at the sets $A_{(b)}:=A\setminus H_b$ and the functions $M_{(b)}(\bullet):=A(\bullet)-A_{(b)}(\bullet)$:

1) $M_{(p)}M_{(q)}=0$ identically, once $p\ne q$, by definition of $M_{(b)}$;

2) $\hat A_{(2b)}(\bullet)=2A(\bullet)-M_{(b)}(\bullet)$, and $\hat A_{(2b+1)}(\bullet)=2A(\bullet)$, by definition of $\hat A_{(b)}$;

3) $g(\bullet)=2A(\bullet)-\max_b M_{(b)}(\bullet)=2A(\bullet)-\sum_b M_{(b)}(\bullet)$ by (2) and then (1).

\vspace{1ex}

{\it Proof of Lemma.} Let $p$ and $q$ be the the minimal and maximal value of the first coordinate on the set $A^\gamma$.

If $p=q$, then, near $\gamma$, we have $\hat A_{(b)}(\bullet)=2A(\bullet)$ unless $b=2p$. This gives $g(\bullet)=\hat A_{(2p)}(\bullet)$ by (3) and (2).

Otherwise (for $p<q$), the support face $\hat A^\gamma$ belongs to the 1-dimensional convex hull of $2A^\gamma$, and thus near $\gamma$ we have $\hat A_{(b)}(\bullet)=2A(\bullet)$ unless $b=2p$ or $2q$. This gives $g(\bullet)=\min(\hat A_{(2p)}(\bullet),\hat A_{(2q)}(\bullet))$  by (3) and (2). \hfill$\square$

\vspace{1ex}

{\it Proof of Theorem \ref{hatunicorn}.1.} {\bf I.} The function $g$ is locally convex near every $\gamma\notin S$, because $$g(\bullet)=\min(\hat A_{(2p)}(\bullet),\hat A_{(2q)}(\bullet))=2A(\bullet)-\max(M_{(2p)}(\bullet),M_{(2q)}(\bullet))=$$ $$=2A(\bullet)-M_{(2p)}(\bullet)-M_{(2q)}(\bullet)=A_{(p)}(\bullet)+A_{(q)}(\bullet)\mbox{  near } \gamma.$$
Here the equalities follow by the lemma, (2), (1) and the definition of $A_{(b)}$ repsectively.

{\bf II.} Thus, by the removable disvexity theorem, $g$ is a convex cone-wise linear function. Thus it is the support function of $\hat A$. The latter is a lattice polytope, because $g$ is integer valued. \hfill$\square$

\begin{lemma}\label{l2dim}
If $A\subset\Z^n$ is not contained in a line, and the first coordinate takes more than two values on it, then $\hat A$ is not contained in a line. 
\end{lemma}
\begin{proof} Choose a segment $I$ connecting points of $A$ with minimal and maximal value of the first coordinate, and a point $a+v\in A$ for $a\in I$ and non-zero $v\in H_0$. Then the set $(A+A)\setminus (I+I)$ has at least three points of the form $a_i+\varepsilon_i\cdot v, \varepsilon_i>0, a_i\in (I+I)\cap H_{b_i}, b_1<b_2<b_3$.
Thus, for every $b\in\Z$, the polytope $\hat A_{(b)}$ contains the point $a_2+\varepsilon\cdot  v$ for small enough $\varepsilon>0$. Thus $\hat A$ contains both $a_2+\varepsilon\cdot v$ and a subsegment of $I+I$, including $a_2$, thus $\dim\hat A>1$.
\end{proof}

{\it Proof of Theorem \ref{hatunicorn}.2} resembles the proof of lemma. 

{\bf I.} Assume that the first coordinate takes one value $p$ on $A^\gamma$. Then, near $\gamma$, we have $\hat A_{(b)}(\bullet)=2A(\bullet)$ unless $b=2p$. This gives $g(\bullet)=\hat A_{(2p)}(\bullet)=A_{(p)}(\bullet)+A(\bullet)$ by (3) and (2).

Thus $\hat A^\gamma=A_{(p)}^\gamma+\conv A^\gamma$, and $\conv A^\gamma$ is a segment parallel to $\hat A^\gamma$.

{\bf II.}  Assume that the first coordinate takes two values $p$ and $q$ on $A^\gamma$. Then $\hat A^\gamma$ equals $\conv (A\cap H_p)+\conv(A\cap H_q)$. Since the former is a segment, both $\conv (A\cap H_p)$ and $\conv(A\cap H_q)$ are segments parallel to it, i.e. $\conv(A^\gamma)$ is a trapezoid whose bases sum up to $\hat A^\gamma$.

{\bf III.}  Assume that the first coordinate takes more than two values on $A^\gamma$. Then, by Lemma \ref{l2dim}
applied to $\hat A^\gamma$, the set $\conv A^\gamma$ is a segment parallel to $\hat A^\gamma$.

\begin{sledst}\label{corolcritunic} The polytope $\hat A$ is compatible with the fan $[A]([\hat A]-[A])$.
\end{sledst}
\begin{proof}
By definition, we need to prove that the fan $[A][\hat A]-[A][A]$ contains no $\gamma$ from dual cones of vertices or edges of $\hat A$, i.e. such that $\hat A^\gamma$ is a vertex or edge. 

If $\hat A^\gamma$ is a vertex, then $\gamma$ is not in $[\hat A]$, all the more not in $[A][\hat A]-[A][A]$.

If $\hat A^\gamma$ and $\conv A^\gamma$ are parallel segments, then $\gamma$ is not in $[A][\hat A]$ or $[A][A]$.

Otherwise Theorem \ref{hatunicorn}.2 implies that $\gamma$ belongs to the interior of  codimension 2 cones of both $[A][\hat A]$ and $[A][A]$, having the same multiplicity. Then $\gamma$ is not in their difference.
\end{proof}

\subsection{The symmetric incremental}

Let $I:\Z^n\to\Z^n$ permute the first two coordinates, denote the difference of these coordinates by $d:\Z^n\to\Z$, and by $\mathcal{I}$ the segment connecting $(1,0,0,\ldots,0)$ and $(0, 1,0,\ldots,0)$. As always, $A\subset\Z^n$ is a finite set.
\begin{defin}\label{defdenom}
1. For $b\in\Z$, define $\check A_{(b)}:=\conv (A\setminus\{d=b\}+IA)\cup I(A\setminus\{d=b\}+IA)$.

2. Let the {\it denominator} $d_A$ be the GCD of the pairwise differences of the values $d(A)$. 

\end{defin}
Note that $d_A$ is defined unless $A$ can be shifted to $\{d=0\}$, the fixed hyperplane of $I$.
\begin{theor}\label{checkunicorn}
1. $\check A:=\conv ((A+IA)\setminus\{d=0\})$ has $d_A\mathcal{I}$ as a Minkowski summand.

2. The support function of $\check A$ is the minimum of those of $\check A_{(b)}$ over all $b\in\Z$.

3. The polytope $\check A$ is compatible with the fan $[A]([\check A]-[A])$.

4. The Minkowski difference $\check A-d_A\mathcal{I}$ is compatible with $[A]([\check A-d_A\mathcal{I}]-[A])$.
\end{theor}
We prove only part 2: part 1 is clear, and parts 3\&4 are deduced from 2 in the same fashion as Theorem \ref{checkunicorn}.2 and Corollary \ref{corolcritunic} from Theorem \ref{checkunicorn}.1.
\begin{proof}
The sought equality $\check A(\gamma)=\min_b\check A_{(b)}(\gamma)$ for $\gamma\in(\Z^n)^*$ can be reduced to the case $n=2$ by projecting $A$ along $\{d=0\}\cap\ker\gamma$. So we assume $n=2$, and choose the points $a\in A^\gamma$ and $a'\in IA^\gamma$ with highest possible values of $d$. We have two cases.

1) Assume $d(a)+d(a')\ne 0$, then $a+a'$ is a point of both $\check A$ and all $\check A_{(b)}$, at which $\gamma$ attains its maximum. This proves the sought equality.

2) If $d(a)+d(a')=0$, then $a'=Ia$, and $a+a'$ is a vertex of $\conv(A+IA)$, however it belongs neither to $\check A$ nor to $\check A_{(\pm d(a))}$. In this case we choose the points $q\in (A\setminus\{d=d(a)\})^\gamma$ and $q'\in (IA\setminus\{d=d(a')\})^\gamma$ with highest possible values of $d$.

Assume wlog that $\gamma(q-a)\geqslant\gamma(q'-a')$ (otherwise applying the symmetry reverses the inequality). Then $q+a'$ is a point of both $\check A$ and all $\check A_{(b)}$, at which $\gamma$ attains its maximum. This proves the sought equality.
\end{proof}

\subsection{Mixed volumes of incrementals}

Since $\hat A$ and $\check A$ for $A\in\Z^n$ are incremental polytopes for some complete intersections, we care how they evaluate the quantity from Theorem \ref{thcatchup}.2.

For $A\subset\Z^3$, it can be done in the language of link polygons $[A\cap\ldots/E]$ and $[A-\ldots/E]$ (Definition \ref{defpolylink}) for certain special edges $E\subset A$.
\begin{defin}\label{defblinder}
An edge $E$ is called {\it horizontal}, if the first coordinate is constant on it. 

It is called a {\it blinder}, if $E+IE$ is an edge of $A+IA$ (equivalently, if it is a preimage of an edge of $pA$, which is itself an edge contained in a plane $\{d=b\},\,b\in\Z$).
\end{defin}
Let $\overline b$ and $\underline b$ be the largest and smallest value of the first coordinate on $A$, $p$ the projection along the antiinvariant line of the involution $I$, and $e:=e_n$ the polynomial from Theorem \ref{bkk1}.
\begin{utver}\label{answerlink3}$e(A,\hat A-A)$, $e(\mathcal{I},A,\check A-d_A\mathcal{I}-A)$ and $e(A,\check A-d_A\mathcal{I}-A)$ equal

$$\MV(A,A\setminus H_{\underline b},A+A\setminus H_{\underline b})+\MV(A,A\setminus H_{\overline b},A+A\setminus H_{\overline b})-2\Vol A -\sum_{b\in(\underline b,\overline b)}\sum_{\mbox{\scriptsize edge }E\subset H_b}\Vol E\cdot \Vol [A\cap H_b/E],$$

$$\Vol pA-\sum\limits_{b\in\Z}\sum\limits_{\mbox{\scriptsize blinder }E\subset \{d=b\}}\Vol E\cdot\Vol[pA-p(A\setminus \{d=b\})/E], \eqno{(*)}$$

$$\mbox{ and }\Vol(A+IA)-d_A\cdot (*)-
\sum\limits_{b\in\Z}\sum\limits_{\mbox{\scriptsize blinder }E\subset \{d=b\}} \Vol E \cdot \Vol[(A+IA)\cap\{d=0\}/E]
.$$

\end{utver}
This lengthy form of the answer is the most convenient when it comes to explicit computations. For instance, it is one of key tools in the proof of irreducibility of symmetric complete intersection curves in \cite{symm}. 
For the proof, apply Proposition \ref{locvol} to every term in the sums of the following lemma (which is valid any dimension $n$).
\begin{lemma}\label{lmvincr} 
We have
$e(A,\hat A-A)=e(A,A)-\sum\nolimits_{b\in\Z} (e(A,A)-e(A,\hat A_{(b)}-A)),$
$$e(\mathcal{I},A,\check A-d_A\mathcal{I}-A)=e_{n-1}(pA,pA)-\sum\nolimits_{b\in\Z} (e_{n-1}(pA,pA)-e_{n-1}(pA,p\check A_{(b)}-pA)),\mbox{ and }$$
$$e(A,\check A-d_A\mathcal{I}-A)=e(A,IA-d_A\mathcal{I})-\sum\nolimits_{b\in\Z} (e(A,IA-d_A\mathcal{I})-e(A,\check A_{(b)}-d_A\mathcal{I}-A)).$$

\end{lemma}
\begin{proof}
Part 1: apply Corollary \ref{locdeltalin} and its subsequent remark to $d=n$, $F=\R^n$, $p=2$, and  $\Phi(m_1,m_2)=e_n(m_1,m_2-m_1)$, with $m_1=\tilde m_1$, $m_2$ and $\tilde m_2$ the support functions of $A$, $\hat A$ and $2A$ respectively. It gives the sought equality, because $\tilde m_{2,\lambda}$ are the support functions of those $\hat A_b$ that do not coincide with $\hat A$ by Theorem \ref{hatunicorn}.1.  

Parts 2 and 3: the same for $m_1=\tilde m_1$, $m_2=\tilde m_2$, $m_3$ and $\tilde m_3$ the support functions of $\mathcal{I}$, $A$, $\check A$ and $A+IA$ respectively. The answer given by Corollary \ref{locdeltalin} for Part 2 should be further simplified with the equality of Remark \ref{remmvprod}.

\end{proof}

\section{Engineered complete intersections}\label{scci}

\subsection{Engineered complete intersections are smooth}

In the beginning we promised to study critical complete intersections: the ones whose defining equations are partial derivatives of a given generic polynomial. To do so uniformly, we rise the generality.

For a finite set $A\subset M\simeq\Z^n$, 
define the bilinear product $*:\C^A\times\C^A\to\C^A$ as $$x^a*x^b:=\begin{cases}
x^a \mbox{ if } a=b \\ 0 \mbox{ otherwise,}
\end{cases}$$so that the coefficient of every monomial in the product $f*g$ is the product of the coefficients of the same monomial in the polynomials $f$ and $g$.
\begin{exa}
The following operations on Laurent polynomials $f\in\C^A$ can be represented as $v*f$ for an appropriate $v\in\Q^A\subset\C^A$:

-- taking a partial derivative or antiderivative (of any order) of $f$;

-- taking a homogeneous component of $f$ with respect to any valuation on the lattice $M$;

-- any linear combination or composition of these operations.
\end{exa}

\begin{defin} 1. If $A\subset M\simeq\Z^n$ is the set of monomials of a Laurent polynomial $f$, and $v_1,\ldots,v_k\in\C^A$ are linearly independent, then the system of equations $v_1*f=cdots=v_k*f=0$, is called an {\it engineered complete intersection}.

2. More generally, if $S_1,\ldots,S_m$ are engineered complete intersections, and $g_1,\ldots,g_m\in T\simeq\CC^n$ are generic, then the $\bigcap_i g_i\cdot S_i$ is called an {\it $m$-engineered complete intersection}.

\end{defin}
For $i=1,\ldots,m$, let $A_i\subset M$ be the set of monomials of $f_i$, let $v_{i,j}\in\C^{A_i}$ be linearly independent, and $g_i\in T\simeq\CC^n$ generic (note that $g_1$ can be set to $1\in T$ with no loss in generality, which is especially convenient for $m=1$). We study the $m$-engineered complete intersection $$v_{i,j}*f_i(g_i\cdot x)=0,\,i=1,\ldots,m.\eqno{(*)}$$
\begin{utver}\label{propgenccireg}
For generic $(f_1,\ldots,f_m)\in\C^{\mathcal{A}}:=\C^{A_1}\oplus\cdots\oplus\C^{A_m}$, this system of equations $(*)$ defines a regular complete intersection in the torus $T\simeq\CC^n$.
\end{utver} 
\begin{proof} We first assume that $m=1$, 
so that the system of equations of interest has the form $v_j*f(x)=0,\, j=1,\ldots,k,\, f\in\C^A$. In this case we find a nondegenerate $k\times k$ complex matrix $C$ and $a_1,\ldots,a_k\in A$ such that the tuple $\tilde v:=v\cdot C$ is ``unitriangular'': the coefficient of the monomial $x^{a_i}$ in $\tilde v_j$ is 0 for all $j<i$ and 1 for $j=i$.

Define $$L_i\subset\underbrace{\C^A\oplus\cdots\oplus\C^A}_{i}$$ as the subspace 
of tuples $(\tilde v_1*f,\ldots,\tilde v_i*f)$ for all $f\in\C^A$. We observe that Proposition \ref{propregind} is applicable to $L:=L_i,\, a:=a_i$ and $L':=L_{i-1}$ (with $L_0:=\{0\}$), thanks to the ``unitriangularity'' of $\tilde v$. Thus we conclude by induction on $i$ that $L_i$ is regular for the torus $T$. This is what we need: for generic $f\in\C^A$, the tuple $\tilde v*f$ is a generic point of $L_k$, thus the regularity of $L_k$ for $T$ means that $\bigl(\tilde v*f\bigr)(x)=0$ defines a regular complete intersection in the torus $T$, thus so does $\bigl(v*f\bigr)(x)=\bigl(f*\tilde vC^{-1}\bigr)(x)=0$.

\vspace{1ex}

We now come back to arbitrary $m$ in the definition of the engineered complete intersection. 

For regular subvarieties $V_i$ in the algebraic group $T$, almost any choice of elements $g_i\in T$ ensures that the shifted copies $g_i\cdot V_i$ intersect transversally. 
We shall apply this fact to $V_i$ defined by the equations ${v_{i,j}}*f_i(x)=0$. Since every $V_i$ is a regular complete intersection (this is shown in the first part of the proof), we conclude that the equations
$${v_{i,j}}*f_i(g_i\cdot x)=0,\,i=1,\ldots,m,$$
define a regular complete intersection for generic $f_i\in\C^{A_i}$ and $g_i\in T$.
This is equivalent to the statement of the proposition, because the map $$T\times\C^{A_i}\to\C^{A_i},\,(g_i,f_i)\mapsto f_i(g_i\cdot x),$$is dominant.
\end{proof}

\subsection{Engineered complete intersections are NUC} To prove it, we first introduce notation to conveniently describe the cancellation matrices $C_l,\,l\in M^*$, for an engineered complete intersection. 
\begin{defin}\label{defsuppseq}
Assume that finite $A\subset M\simeq\Z^n$ is not contained in a hyperplane, $l\in M^*$ is non-zero, and $v_1,\ldots,v_k\in C^A$ are linearly independent.

1. A sequence $\varphi:[1..k]\to\Z$ is said to be {\it $l$-linearly independent}, if, for every $c\in\Z$, the restrictions of the functions $v_i,\, \varphi(i)\ge c$, to $$A_{\geqslant c}:=A\cap\{a\,|\,l(a)\geqslant c\}$$ are linearly independent.

2. A sequence $\varphi:[1..k]\to\Z$ is said to be {\it $l$-support}, if it satisfies one of the following equivalent conditions:

a) it is $l$-linearly independent, but increasing its value at $i$ makes its restriction to $[1..i]$ $l$-linearly dependent, for every $i\leqslant k$;

b) its restriction to $[1..k-1]$ is $l$-support, and it is $l$-linearly independent, but increasing its value at $k$ makes it $l$-linearly dependent.

We shall denote the $l$-support sequence by $\varphi_l$, thanks to the following fact.
\end{defin}
\begin{utver}\label{existslsupp}
An $l$-support sequence exists, and it is unique.
\end{utver}
\begin{proof}
We shall use the version (b) of the definition, and assume by induction that the $l$-support restriction of $\varphi$ to $[1..k-1]$ exists and is unique.

Now define $\varphi(k)$ to be the maximal $C\in\Z$ for which $v_1,\ldots,v_k$ are $l$-linearly independent on $A_{\geqslant C}$. Such $C$ does exist, because $v_1,\ldots,v_k$ are linearly independent on $A_{\geqslant c}$ for small enough $c$, but not for large enough $c$.

It remains to notice that $\varphi$ defined in this way is $l$-linearly independent, but loses this property once we increase its value at $k$. The latter fact is tautological from the choice of $C$, and the former one is straightforward (though not tautological: we still have to verify the condition in Definition \ref{defsuppseq}.1 for all $c<C$).
\end{proof}
Since $({\sum c_iv_i})*f=\sum c_i ({v_i}*f)$, Proposition \ref{existslsupp} can be translated into the language of Laurent polynomials as follows.
\begin{sledst}\label{critcancellations}
If $A\subset M$ is the set of monomials of a Laurent polynomial $f$, then there exists a complex unitrianglular matrix $C_l$ such that the $l$-degrees of the tuple $\tilde f:=C_l\cdot (v_1*f,\ldots,v_k*f)$ form the $l$-support sequence: $$\deg_l \tilde f_i=\varphi_l(i).$$
\end{sledst}
\begin{rem}
1. The proof of Proposition \ref{existslsupp} is constructive, giving an algorithm to compute the matrix $C_l$ by induction on $k$, starting from $v_1,\ldots,v_k$ and $A$.

2. Such a numerical matrix $C_l$ is unique, up to adding to its $i$-column a multiple of its $j$-column for any $j<i$ such that $\varphi_l(j)\leqslant\varphi_l(i)$.

3. In the definition of nondegeneracy with cancellations we let the matrix $C_l$ be polynomial rather than numerical, and then it is not unique in the sense (2), unless $l$ belongs to the tropical fan of the engineered complete intersection (see Corollary \ref{tropcancel}).
\end{rem}
\begin{theor}\label{thccinondegexplic}
1. If $A\subset M$ is the set of monomials of $f$, then the engineered complete intersection $v_1*f=\cdots=v_k*f=0$ is cancellable with cancellations $C_l$ defined by Corollary \ref{critcancellations}.

2. For a generic Laurent polynomial $f\in\C^A$, it is nondegenerate upon these cancellations.
\end{theor}
\begin{proof}
For $c\in\Phi_l:=\varphi_l[1..k]$, denote $A\cap l^{-1}(c)$ by $A_c$ and the restriction of $f$ to $A_c$ by $f_c$. We need to prove for generic $f_c\in\C^{A_c}$ and all $i\in[1..k]$ that the system of equations
$${v_j}*f_c(x)=0,\,c\in\Phi_l,\,j\in\varphi_l^{-1}(c)\cap[1..i]$$defines a regular complete intersection in the algebraic torus $T\simeq\CC^n$. 
This is so by Proposition \ref{propgenccireg}, because this system of equations itself defines an engineered complete intersection.\end{proof}
\begin{sledst}\label{corolccinondeg}
The complete intersection $(*)$ is cancellable, and NUC for generic $f$.
\end{sledst}

\subsection{Geometry of engineered complete intersection.}\label{scancelcrit}{$ $}

{\it Proof of Theorem \ref{th0cci} on geometry of $S_1=\{f=\partial f/\partial x_1=0\}$.}
The complete intersection $S_1$ is smooth by Proposition \ref{propgenccireg} and NUC by Theorem \ref{thccinondegexplic} for generic $f\in\C^A$.
Comparing Theorem \ref{hatunicorn} to Corollary \ref{critcancellations} for this complete intersection, we find that $A$ and $\hat A$ are its incremental polytopes: $\phi_l(1)=A(l)$ and $\phi_l(1)+\phi_l(2)=\min_b\hat A_{(b)}(l)= \hat A(l)$. 

The latter polytope is 
compatible with the tropical fan $[A]([\hat A]-[A])$ by Corollary \ref{corolcritunic}. This fan is the tropicalization of $S_1$ by Theorem \ref{thcatchup}.5 for all $f\in\C^n$ such that $S_1$ is NUC. 

Thus, all such smooth $S_1$ have the same diffeomorphism type 
by Proposition 2.4 
in \cite{schon}.
Irreducibility, Calabi--Yau, and the Euler characteristics are deduced from Theorem \ref{thcatchup}. \hfill $\square$

\vspace{1ex}

This does not extend to all engineered complete intersections, because they may not have incremental polytopes, so Theorem \ref{thcatchup} may be not applicable to them. Our plan is as follows:

1) Formulate a general version of Theorem \ref{thcatchup} that works for all NUC complete intersections. Instead of the incremental polytopes (which may not exist), the combinatorial input in the general theorem will be the {\it tropicalization} of a cancellable complete intersection.

2) Find the tropicalization of an engineered complete intersection.

\begin{utver} \label{tropcancel}
If a complete intersection $f_1=\cdots=f_k=0$ in the torus $T\simeq\CC^n$ with the character lattice $M\simeq\Z^n$ is cancellable with cancellations $C_l$, then, denoting $\tilde f:=C_l\cdot f$, the following objects are well defined (i.e. depend only on $f$ and not on the choice of $C_l$):

1. The set $F_i$ of all $l\in M^*$ such that the system of equations $\tilde f_1^l=\cdots=\tilde f_i^l=0$ has a solution.

2. For every such $l$, the $l$-degree of the cancelled equation
$\tilde f_{i+1}$, denoted by $m_{i+1}(l)$.

3. The collection of fans $F_i$ and piecewise linear functions $m_i:F_{i-1}\to\R$ is a tropical complete intersection (in the sense of \cite{schon}, Definition 1.12).

4. The latter is the tropicalization of $f=0$ in the sense of Definition 1.8 
in \cite{schon}.
\end{utver}
It will be proved below together with Theorem \ref{thnewtcinondegexplic}.
\begin{rem}
Part (4) means that the fans $F_{i}$ and functions $m_i$ on them are the topicalizations of the sets $f_1=\cdots=f_{i}=0$ and the functions $f_i$ on them. The functions $m_i$ determine the fans $F_i$ (being the corner locus of $m_i$), see Remark 1.13 
in \cite{schon}. 
\end{rem}

As a consequence, we can apply the results of \cite{schon} on sch\"on and Newtonian complete intersections to a NUC complete intersection $f=0$. Let $(F_i,m_i)_{i=1,\ldots,k}$ be its tropicalization, as described by Proposition \ref{tropcancel}. Applying Theorems 1.5, 1.17 

and Example 1.15 
in \cite{schon}, we have the following. Recall that the Newton polyhedron $N_i$ of $f=0$ is the intersection of halfspaces $\{a\,|\,l(a)\leqslant m_i(l)\}$ over all non-zero $l\in M^*$.

\begin{theor} \label{thcatchup1} In this setting:
1. the set $\{f=0\}$ has at most isolated singularities.

2. If the polyhedra $N_1,\ldots,N_k$ have full dimension, and $k<n$, 
then $\{f=0\}$ is connected (and, moreover, its first $n-k$ Betti numbers equal those of $T$).

3. In particular, in this case $\{f=0\}$ is irreducible, unless it is a singular curve.

4. The Euler characteristics of $\{f=0\}$ equals 
$$\frac{\delta^n}{n!} \left(\frac{m_1}{1+m_1}\cdots\frac{m_k}{1+m_k}\right)+(-1)^{n-k}(\mbox{the sum of the Milnor numbers of the singularities, if any}).$$
Here $\delta$ is the corner locus, see Remark 1.18.1-3 
in \cite{schon} for details.

5. Assume that the function $\sum_i m_i$ extends from 
the fan $F_k$ to the support function $m:M^*\to\Z$ of a convex polytope $P$, and $F_{k}$ is a union of certain cones $S_i$ from the dual fan of $P$, each of whom as a semigroup $S_i\cap M^*$ is generated by lattice points on the boundary of the dual polytope $\{a\,|\, m(a)=1\}$.
Then the $P$-toric compactification of a smooth $f=0$ is Calabi--Yau.
\end{theor}
{\it Proof of Theorem \ref{thcatchup}.}
If the incremental polytopes $A_j$ exist, then its support function $A_j(\bullet)$ equals $m_1+\cdots+m_j$, so $m_i(l)=A_i(l)-A_{i-1}(l)$. Plug this into the preceding theorem.
\hfill$\square$

\vspace{1ex}

We now summarize how the results of this section describe the geometry of engineered complete intersections.
\begin{defin}\label{defteci}

For finite $A\subset M$ and linearly independent $v_1,\ldots,v_k\in\C^A$, the {\it tropical engineered complete intersection} is the sequence $(F_i,m_i)_{i=1,\ldots,k}$, in which $F_i$ is a tropical fan in the space $M^*\otimes\R=:F_0$, defined as the corner locus of $m_i$, and $m_i(l)$ is a piecewise linear function on $F_{i-1}$, defined as $\varphi_l(i)$ (the latter is introduced in Corollary \ref{critcancellations}).

\end{defin}
\begin{utver}\label{remmainengin} Let $A\subset M$ be the set of monomials of $f$.

1. The tropicalization of the engineered complete intersection $S$ given by the equations $f_i:=v_i*f=0$, is the tropical engineered intersection, associated to $(A,v_1,\ldots,v_k)$ as above.

2. For generic $f\in\C^A$, the complete intersection $S$ is smooth and NUC, and all such complete intersections are diffeomerphic.

3. Plugging the tropicalization $m_i(l)=\varphi_l(i)$ into Theorem \ref{thcatchup1}, we learn about connectedness, Euler characteristics and Calabi--Yau-ness of the engineered complete intersection $S$.
\end{utver}
\begin{proof}
Part 1 is by Corollary \ref{critcancellations}. Part 2 holds true because $S$ is sch\"on by Theorem \ref{thnewtcinondegexplic}, sch\"on varieties do not change topology under deformations preserving the tropical fan, by Proposition 2.4 
in \cite{schon}, and the tropical fan of $S$ is constant by Part (1).
\end{proof}

\section{Symmetric complete intersections}\label{ssci}

\subsection{Symmetric complete intersections are smooth}

Let $A$ be a finite set in the monomial lattice $\Z^n$ of the algebraic torus $\CC^n$ with coordinates $x_1,\ldots,x_n$, let $I:\CC^n\to\CC^n$ interchange $x_1$ and $x_2$, let $f\in\C^A$ be a generic polynomial, and let $d_A$ be the denominator of $A$ (Definition \ref{defdenom}). We now interpret it as the largest integer such that
$f-f\circ I$ is divisible by $x_1^{d_A}-x_2^{d_A}$ in the case $0\in A$, and denote the quotient by $F$.

\begin{utver}\label{proptransv} 
If $A\ni 0$ is not contained in the fixed hyperplane of the involution $I$, then, for a generic $f\in\C^A$, the hypersurfaces $\{f=0\}$, $\{F=0\}$ 
and $\{x_1^{d_A}=x_2^{d_A}\}$ are smooth and mutually transversal.

In particular, the algebraic set $\{f(x)=f(Ix)=0\}$ splits into a smooth hypersurface in $\{x_1^{d_A}=x_2^{d_A}\}$ and one more smooth codimension 2 set transversal to it in $\CC^n$.
\end{utver}
\begin{rem}
The first statement makes no sense without the assumption $0\in A$ (because $x_1^{d_A}-x_2^{d_A}$ may not divide $f(x)-f(Ix)$ in this case). However, the second one not only makes sense for $0\notin A$, but is obviously invariant under translations of $A$. Thus it is valid without the assumption $0\in A$.
\end{rem}
{\it Proof.} 
For smoothness of the hypersurface $F=0$, 
apply Lemma \ref{fdiscr} to the regular functions 
$F,\, \mu\in A\setminus\{0\}$, on $(\C^*)^n$. For its transversality to $\{x_1^{d_A}=x_2^{d_A}\}$, apply the same lemma to the restrictions of the same functions to the smooth variety $\{x_1^{d_A}=x_2^{d_A}\}$.

We have proved that, for generic $f\in\C^{A\setminus\{0\}}$, the hypersurfaces $$
\{F=0\} \mbox{ and } \{x_1^{d_A}=x_2^{d_A}\}\eqno{(*)}$$are smooth and transversal. It now remains to show that, for such a generic $f\in\C^{A\setminus\{0\}}$ and generic $c\in\C$, the difference $f-c\in\C^A$ satisfies the statement of the lemma.

Note that changing $f$ to $f-c$ does not affect the hypersurfaces $(*)$, so it is enough to prove that 

$c$ is a regular value for $f$ and for its restrictions to the varieties $(*)$ and to their intersection. All but finitely many values of $c$ are regular by the Bertini-Sard theorem.

\hfill$\square$

\subsection{Symmetric complete intersections are often NUC}\label{ssymmnuc}

For $0\in A\subset\Z^n$ with the involution $I:\Z^n\to\Z^n$ swapping the first two coordinates, we have introduced the denominator $d_A$, the 
segment $\mathcal{I}$, and the polytopes $A_{(b)}$ (Definition \ref{defdenom}), whose intersection is the symmetric incremental $\check A:=conv((A+IA)\setminus\{d=0\})$ (Theorem \ref{checkunicorn}). We now additionally assume $$d_A=d_{\Gamma\cap(A+IA)\setminus\{d=0\}} \mbox{ for every symmetric face }\Gamma=I\Gamma\subset\check A,\,\dim\Gamma>2.\eqno{(*)}$$

\begin{defin}\label{defrestr}
1. For a set $A\subset M\simeq\Z^n$ and $l\in M^*$, we denote $A(l):=\max l(A)$, and define the {\it support face} $A^l$ as $\{a\in A\,|\,l(a)=A(l)\}$.

2. For $f(x)=\sum_{a\in A}c_ax^a$, we define its {\it restriction} to $B\subset A$ as $f|_B:=\sum_{a\in B}c_ax^a$.
\end{defin}
\begin{utver}\label{thsymmincrem}
1. For any $A\subset\Z^n$, generic $f\in\C^A$ and generic $c\in\C$ including all $d_A$-roots of unity, the complete intersection $x_1-cx_2=f=F=0$ is NUC with cancellation matrices as in the table below and incremental polytopes $\mathcal{I}$, $A+\mathcal{I}$ and $\check A-(d_A-1)\mathcal{I}$.

2. In particular, $x_1^{d_A}-x_2^{d_A}=f=F=0$ is NUC with incrementals $d_A\mathcal{I}$, $A+d_A\mathcal{I}$ and $\check A$. 

3. If $A$ satisfies $(*)$, then $f=F=0$ is NUC with incrementals $A$ and $\check A-d_A\mathcal{I}$.

\end{utver}

\begin{rem}
1. Without condition $(*)$, the system $f=F=0$ may be not NUC, but cancellable once $F$ is defined (i.e. unless $d$ is constant on $A$). The proof is the same as below.

2. If there are no blinders (Definition \ref{defblinder}), the genericity assumption is just nondegeneracy.
\end{rem} 

\begin{proof} Part 2 is a corollary of part 1. For parts 1 and 3, following Definition \ref{defnuc}, to every $l\in(\Z^n)^*$ we associate a suitable cancellation matrix $C_l$, and look at the cancelled system of equations. In the list below, for $l$ satisfying each of the specified conditions, we provide the cancellation matrix $C_l$ and the support sets $A_l'$ and $A_l''$ for the last two equations of the cancelled system. This data is shown for the system of part (1), but chosen in such a way that the last $2\times 2$-minor of the matrix $C_l$ suits part 3 as well (which is not automatic).

{\small
\begin{center}
\begin{tabular}{c c c c} 
\toprule 

& Condition on $l$ & matrix $C_l$ & support sets $A_l'$ and $A_l''$ \\ [0.5ex] 
\bottomrule \toprule

{\bf I} & \begin{tabular}[x]{@{}c@{}} $l(v)=0$ for $v:=(1,-1,0,\ldots,0)$, \\ but $d$ does not vanish on $(A+IA)^l$\end{tabular} &
$ \begin{bmatrix}
1 & 0 & 0 \\
0 & 1 & 0 \\
0 & 0 & 1
\end{bmatrix}$ & \begin{tabular}[x]{@{}c@{}} $A^l$, \\  $(\check A-d_A\mathcal{I})^l$ \end{tabular} \\ 
\midrule

{\bf II} & \begin{tabular}[x]{@{}c@{}} $d$ vanishes on $(A+IA)^l$, \\ and thus is a constant $d^l$ on $A^l$\end{tabular} &
$ \begin{bmatrix}
1 & 0 & 0 \\
0 & 1 & -x_1^{-d^l}(x_1^{d^l}-x_2^{d^l})/(x_1^{d}-x_2^{d}) \\
0 & 0 & 1
\end{bmatrix}$ & \begin{tabular}[x]{@{}c@{}} $A^l$, \\  $(\check A-d_A\mathcal{I})^l-d^lv$ \end{tabular} \\ 
\midrule

{\bf III} & Not {\bf (II)} and $l(v)<0 $& $ \begin{bmatrix}
1 & 0 & 0 \\
0 & 1 & -1/\bigl(1-(x_1/x_2)^d\bigr) \\
0 & 0 & 1
\end{bmatrix}$ & $A^l,\; (IA)^l$ \\ 
\midrule

{\bf IV} & Not {\bf (II)} and $l(v)>0 $& $ \begin{bmatrix}
1 & 0 & 0 \\
0 & 1 & -1/\bigl(1-(x_2/x_1)^d\bigr) \\
0 & 0 & 1
\end{bmatrix}$ & $A^l,\; (IA)^l$ \\ [1ex] 
\bottomrule
\end{tabular}
\end{center}
}

\vspace{1ex}

The rational function in {\bf (III-IV)} stands for its Taylor polynomial of sufficiently high degree.

It remains to explain why, in each of these cases, the cancelled system of equations is regular for generic $f\in\C^A$.
We denote the respective cancelled system by $G=0$ and $g=0$, for the systems of equations considered in part 1 and 3 of the theorem respectively.

{\bf (I)}: The system $g=0$ is symmetric itself, so its regularity  follows by Proposition \ref{proptransv} applied to $f^l$ in place of $f$, and the same denominator $d_{A^l}=d_A$ thanks to condition $(*)$.

In the absence of condition $(*)$, the system $g=0$ has singular roots $(x_1,\ldots,x_n)$ such that $x_2/x_1$ is a root of unity of degree strictly higher than $d_A$. Thus such roots do not satisfy the first equation $x_1-cx_2=0$ of the system $G=0$, so the system $G=0$ is regular as well.

{\bf (II)} for $l(v)=0$: The second equation of the system $g=0$ is $I$-symmetric, thus its roots $(x_1,\ldots,x_n)$ are regular, as soon as we have condition $(*)$ or unless $x_2/x_1$ is a root of unity of degree strictly higher than $d_A$.

The first equation of the system $g=0$ is independent of the second one, i.e. the tuples $g$ for all $f\in\C^A$ form a vector subspace of the form $\C^{A'_l}\oplus L$ in the space $\C^{A'_l}\oplus\C^{A''_l}$. Thus, by Proposition \ref{propregind}, roots of the system $g=0$ are again regular, unless we have condition $(*)$ or unless $x_2/x_1$ is a root of unity of degree strictly higher than $d_A$.

In the latter case, the singular roots of $g=0$ never satisfy the first equation $x_1-cx_2=0$ of the system $G=0$, so the system $G=0$ is regular as well.

\vspace{1ex}

For the rest of the cases, the first equation of the system $G=0$ has the form (monomial $=0$), making this system inconsistent and thus regular, and the system $g=0$ is a general one with the specified support sets (i.e. the tuples $g$ for all $f\in\C^A$ form the vector space $\C^{A_l'}\oplus\C^{A_l''}$). Thus $g=0$ is regular, regardless of condition$(*)$.

\vspace{1ex}

We have proved that the complete intersections mentioned in the theorem are all NUC. To check that the incremental polytopes are as stated in part 3, we should check that the support functions of these polytopes at every $l$ equal $l(A'_l)$ and $l(A'_l)+l(A''_l)$. This is straightforward from the last column of the table, and the same for part 1. 
\end{proof}

\subsection{Geometry of symmetric complete intersections}

Once $A\ni 0$ is not contained in the invariant hyperplane of the involution $I:\Z^n\to\Z^n$, almost all $f\in\C^A$ define a symmetric complete intersection $f(x_1,x_2,x_3,\ldots,x_n)=f(x_2,x_1,x_3,\ldots,x_n)=0$ which splits into the {\it diagonal} part $x_1^{d_A}-x_2^{d_A}=f=0$ and the {\it proper} part $$S:=\{f=F=0\} \mbox{ with } F:=\bigl(f(x_1,x_2,x_3,\ldots,x_n)-f(x_2,x_1,x_3,\ldots,x_n)\bigr)/(x_1^{d_A}-x_2^{d_A}).$$
The diagonal part for generic $f\in\C^A$ splits into nondenerate hypersurfaces inside $d_A$ hyperplanes of the form $x_1=\sqrt[d_A]{1}\cdot x_2$, so its geometry is described in terms of $A$ by the classical BKK toolkit. We now study the geometry of the proper part in terms of $A$.

Recall that $\mathcal{I}$ is a primitive segment on the antiinvariant line of the involution $I:\Z^n\to\Z^n$ swapping the first two coordinates, and $p$ is the projection along this line.
\begin{theor}\label{th0sci} Assume that $A$ satisfies condition $(*)$ of Section \ref{ssymmnuc}, and $f\in\C^A$ is generic, in the sense that it satisfies conditions of Propositions \ref{proptransv} and \ref{thsymmincrem}.

1. The proper part $S$ of the symmetric complete intersection is smooth and NUC with incremental polytopes $A$ and $\check A-d_A\mathcal{I}$, all such  $S$ are diffeomorphic to each other.

2. Their Euler characteristics equals
$e(A,\check A-d_A\mathcal{I}-A)$; this expression can be simplified as in Proposition \ref{answerlink3}, and further for $A\subset\Z^3$ as in Lemma \ref{lmvincr}.

3. Their tropical fan equals $[A]\cdot([\check A]-d_A[\mathcal{I}]-[A])$ (notation explained in Remark \ref{remcritci}.2).

4. They are irreducible and connected, if $A+d_A\mathcal{I}$ can be shifted to the interior of $\check A$.

5. Their closure in the $\check A$-toric variety are Calabi--Yau, if $\check A-d_A\mathcal{I}$ is reflexive.

6. The singular locus of the symmetric complete intersection $f(x_1,x_2,x_3,\ldots,x_n)=f(x_2,x_1,\allowbreak x_3,\ldots,x_n)=0$ 
consists of $d_A$ components of the form

$x_1-\sqrt[d_A]{1}\cdot x_2=f=F=0$. They are NUC with incremental polytopes $\mathcal{I}$, $A+\mathcal{I}$ and $\check A-(d_A-1)\mathcal{I}$, or just $pA$ and $p\check A$ as a complete intersection in the torus $\{x_1=\sqrt[d_A]{1}\cdot x_2\}$. They are smooth and diffeomorphic to each other.

7. Their Euler characteristics equals
$e(\mathcal{I},A,\check A-d_A\mathcal{I}-A)=e_{n-1}(pA,p\check A-pA)$; this expression can be simplified as in Proposition \ref{answerlink3}, and further for $A\subset\Z^3$ as in Lemma \ref{lmvincr}.

8. Their tropical fan equals $\mathcal{I}\cdot[A]\cdot([\check A]-d_A[\mathcal{I}]-[A])=[pA]\cdot([p\check A]-[pA])$.

9. They are irreducible (equivalently, connected), if $pA$ can be shifted to the interior of $p\check A$.

10. Their closure in the $\check A$-toric variety are Calabi--Yau, if $p\check A$ is reflexive.
\end{theor}

\begin{rem}\label{remsymmci}
1. The irreducibility condition is not a criterion: one can see it from the complete classification of reducible cases for $A\subset\Z^3$ in \cite{symm} We do not know whether this irreducibility condition survives without condition $(*)$. 

2. Parts 3 and 8 are valid without condition (*): their proof (see below) requires the complete intersection to be just cancellable, not NUC, and it is cancellable once $A$ cannot be shifted to the invariant hyperplane of the involution $I$.

3. Parts 5 and 10 are valid without condition (*) as well: in this case, the closure of the complete intersection in the $\check A$-toric variety is not smooth anymore, but its singularities are Gorenstein. This fact is however outside the scope of the present paper.

4. The rest of the statements fail without condition $(*)$.
\end{rem}
\begin{proof}
Parts 1 and 6 follow from  Propositions \ref{proptransv} and \ref{thsymmincrem} (except for the diffeomorphic statements). Then Parts 2-5 and 6-10 follow from  Theorem \ref{thcatchup}. Then the diffeomorphisms follow from Proposition 2.4 
in \cite{schon}.
\end{proof}

\section{NUC vs SCI}\label{scancel}

\begin{theor}\label{thnewtcinondegexplic}
1. If $f_1=f_2=\cdots=f_k=0$ is cancellable (Definition \ref{defnuc} here) then $f_1=f_2|_{\{f_1=0\}}=\cdots=f_k|_{\{f_1=\cdots=f_{k-1}=0\}}=0$ is Newtonian (Definition 1.7 
in \cite{schon}).

2. Being NUC and being Newtonian SCI (Definition 1.2 
in \cite{schon}) is the same thing. 
\end{theor}
\begin{rem}
Note that, in the subsequent proof, the first implication in case (ii) fails without regularity, not allowing to invert statement (1) of the theorem. Indeed, one easily finds a non-cancellable Newtonian complete intersection $\{f_1=f_2=0\subset\CC^2$ (even with integer-valued tropicalizations of the equations).
\end{rem}
\begin{proof} Throughout the proof, we denote by $X_\Sigma\supset T$ a tropical compactification of the complete intersection, i.e. a smooth toric variety  whose every orbit $O$ intersects the closure of $\{f_1=\cdots=f_i=0\}$ by a set of the expected dimension $\dim O-i$ for every $i\leq k$. 

\vspace{1ex}

First, assume that the complete intersection $f_1=\cdots=f_k=0$ in the torus $T$ is Newtonian SCI, and prove that it is NUC. 
Pick non-zero $l\in M^*$ and find the cancellation matrix $C_l$ satisfying Definition \ref{defnuc} of this paper. By induction on the number of equations we may assume that all columns of $C_l$ but the last one are already found. Denoting $\tilde f:=C_l\cdot f$, the equations $\tilde f^l_1=\cdots=\tilde f^l_{k-1}=0$ define a regular complete intersection by the choice of the first columns of $C_l$, and it remains to find $$\tilde f_k=f_k+\sum_{i=1}^{k-1}c_i f_i=f_k+\sum_{i=1}^{k-1}\tilde c_i\tilde f_i,$$such that $\tilde f^l_1=\cdots=\tilde f^l_{k}=0$ is a regular complete intersection too. Here $c_i$ are the entries in the unknown last column of $C_l$. We shall choose them (or, equivalently, $\tilde c_i$) as follows:

1) If $\tilde f_1^l=\cdots=\tilde f_{k-1}^l=0$ defines the empty set in $T$, we are good to go with $c_1=\cdots=c_{k-1}=0$.

2) If this set is not empty, then we choose $c_i$ in such a way that the $l$-degree $\deg_l \tilde f_k$ is as small as possible. Note that it cannot be arbitrarily small unless $\tilde f_1^l=\cdots=\tilde f_{k-1}^l=0$ defines the empty set (considered before) or $f_k$ vanishes on a component of $f_1=\cdots=f_{k-1}=0$, which would mean that $f_1=\cdots=f_{k}=0$ is not a complete intersection.

We shall now prove that the resulting matrix $C_l$ (with the last column $(c_1,\ldots,c_{k-1},1)$ defined as above) fits Definition \ref{defnuc}. 
Let $O$ be the orbit of $X_\Sigma$ whose cone contains $l$ in the relative interior. Its intersection $S$ with the closure of $f_1=\cdots=f_{k-1}=0$ is given by the equations $\tilde f_1^l=\cdots=\tilde f_{k-1}^l=0$. This key statement needs two clarificartions.

a) the functions $\tilde f_j^l$ on $T$ define a set in $O$ in the following sense: $O$ is a quotient torus of $T$, and, up to a monomial factor, $\tilde f^l_i$ can be regarded as a Laurent polynomial on $O$ lifted to $T$.

b) the statement is valid because $\tilde f_1^l=\cdots=\tilde f_{k-1}^l=0$ is a 
complete intersection (by the choice of $\tilde f_i$'s in the second paragraph of the proof).

Now we have the following options for $\tilde f^l_k$: it may vanish on all components of $\tilde f_1^l=\cdots=\tilde f_{k-1}^l=0$, or on some components, or on none. Let us consider each of these possibilities.

i) If $\tilde f^l_k$ vanishes on some components, this means that the corresponding components of the set $S\subset O$ participate in the divisor of zeroes and poles of $f_k|_{C_{k-1}}$ with lower multiplicities, than the other components, contradicting the newtonness of $f=0$.

ii) If $\tilde f^l_k$ vanishes entirely on the regular complete intersection $\tilde f_1^l=\cdots=\tilde f_{k-1}^l=0$, then $\tilde f^l_k$ can be represented as $\sum_{i=1}^{k-1}\alpha_i\tilde f^l_i$, thus $\tilde f_k-\sum_{i=1}^{k-1}\alpha_i\tilde f_i$ has lower $l$-degree than $\tilde f_k$, contradicting our choice of $\tilde f_k$ in the first paragraph of the proof.

iii) We are left with the last possibility: $\tilde f^l_k$ does not vanish on any of the components of $\tilde f_1^l=\cdots=\tilde f_{k-1}^l=0$. Then the equation $\tilde f^l_k=0$ on $S$ defines the intersection of $S$ with the closure of $f_1=\cdots=f_k=0$, which is regular because the latter complete intersection is SCI. Thus $\tilde f_1^l=\cdots=\tilde f_{k}^l=0$ is regular, i.e. $\tilde f=C_l\cdot f$ satisfies the definition of NUC.

\vspace{1ex}

Let us now prove the other implication: assume that the complete intersection $f_1=\cdots=f_k=0$ in the torus $T$ is cancellable or nondegenerate upon cancellations $C_l$. 

Denote $H_i:=\{f_1=\cdots=f_i=0\}$. Pick an interior point $l$ in the cone of $\Sigma$, corresponding to the orbit $O\subset X_\Sigma$, and denote $\tilde f:=C_l\cdot f$. The cancelled equations $\tilde f^l_1=\cdots=\tilde f^l_{i}=0$ define the intersection $\bar H_{i}\cap O$ (in the sense explained in the comment (a) above). If $f_1=\cdots=f_k=0$ is NUC, then, by its definition, the cancelled system is regular, thus $\bar H_{i}$ is smooth and transversal at its intersection with $O$, thus $f_1=\cdots=f_k=0$ is SCI.

Now assume that $f_1=\cdots=f_k=0$ is just cancellable (not necessarily NUC), and deduced that it is Newtonian. The definition of Newtonianity tells us to restrict attention to the case when the orbit $O$ has codimension 1.

In this case, at a generic point $z\in\bar H_{k-1}\cap O$, the function $\tilde f_k|_{\bar H_{k-1}}=f_k|_{\bar H_{k-1}}$ equals $u^{\deg_l\tilde f_l}$ for a suitable reduced local defining equation $\{u=0\}$ of the orbit $O$.

Thus the divisor of zeroes and poles of the function $\tilde f_k|_{\bar H_{k-1}}=f_k|_{\bar H_{k-1}}$ has the same multiplicity at every its point $z$: this multiplicity equals that of $u_{\bar H_{k-1}}$ times the $l$-degree $\deg_l \tilde f_k$. This means $f_1=\cdots=f_k=0$ is Newtonian, with the Newton data
$$m_i(l)=\deg_l(C_l\cdot f)_i.\eqno{(*)}$$

The latter conclusion at the same time proves Proposition \ref{tropcancel}.
\end{proof}

\section{Appendix: the proof of Theorem \ref{thconcr2}.}\label{proofth3} 
For $A\subset M\simeq\Z^n$ and non-zero $l\in M^*$, let $A(l)$ and $\bar A(l)$ be the highest and the second highest value of $l$ on $A$, and
let $A^l$ and $\bar A^l$ be the intersections of $A$ with $\{l=A(l)\}$ and $\{l=\bar A(l)\}$ respectively. We denote restrictions $f|_{A^l}$ and $f|_{\bar A^l}$ (see Definition \ref{defrestr}) as $f^l$ and $\bar f^l$.
\begin{rem}\label{remgenerdf} 1. A sufficient genericity assumption for Theorem \ref{thconcr2}.1 is: for every $l\in M^*$ such that the affine span of $A^l$ is parallel to the first coordinate axis, the system of equations $f^l=\partial f^l/\partial x_2=\cdots=\partial f^l/\partial x_n=\partial \bar f^l/\partial x_2=\cdots=\partial \bar f^l/\partial x_n=0$ has no solutions.

2. A sufficient genericity assumption for Theorem \ref{thconcr2}.2 is: for every $l\in M^*$ such that the affine span of $A^l$ contains 0, the system of equations $df^l=d\bar f^l=0$
has no solutions.

3. These assumptions are not necessary (in contrast to the nondegeneracy assumption for the BKK formula). It would be interesting to find the weakest genericity condition for this theorem (i.e. the hypersurface $D\subset\C^A$ such that it holds for $f\notin D$ and fails for $f\in D$).
\end{rem}

We prove the second part of the theorem; for the first one, the proof is the same with obvious alterations (or can be reduced to the first part by replacing $A\subset\Z^n$ with $\{1\}\times A\subset\Z\oplus\Z^n$).

\subsection{Proof of Theorem \ref{thconcr2} for $\{df=0\}$.}

Given a torus $T\simeq\CC^n$ with the character lattice $M\simeq\Z^n$, we need to compute the number $\sharp A$ of critical points of a generic Laurent polynomial $f\in\C^A$ on $T$, for finite $A\subset M$ not containing 0 and not contained in a hyperplane. 

Note that, if the linear span of $A$ is $\hat M\subset M$ of smaller dimension $\hat n$, then $f$ has either infinitely many or no critical points, so $\sharp A$ as defined in the statement of theorem makes little sense. However, in this case $f$ is the pull back of a polynomial $\hat f$ on the $\hat n$-dimensional torus $\hat T$ under the projection $T\mapsto\hat T$ corresponding to the embedding $\hat M\hookrightarrow M$ of their character lattices. In this case we define $\sharp A$ as the number of critical points of $\hat f:\hat T\to\C$.

We choose a generic coordinate system $M\to\Z^n$ in the sense that its last coordinate line $L$ for every $l$ is not contained in the vector span of $A^l\cup\bar A^l$ unless the latter is the whole $\R^n$, and is not contained in $\{l=0\}$ otherwise.
Denoting by $(x_1,\ldots,x_n)$ the repsective coordinates on the complex torus $T$, 

this ensures that the system of equations $\partial f/\partial x_1=\cdots=\partial f/\partial x_n=0$ is not only an engineered complete intersection, but the degree function $\varphi_l$ (see Corollary \ref{critcancellations}) satisfies 
$$\varphi_l(1)=\cdots=\varphi_l(k)=A(l)>\varphi_l(k+1)=\cdots=\varphi_l(n)=\bar A(l),$$
where $k$ is the dimension of the vector span of $A^l$.

Applying Theorem 4.3 
of \cite{schon} to 
$$\partial f/\partial x_1=\cdots=\partial f/\partial x_n=0$$with $P_i=A$ and $m_i(l)=\varphi_l(i)$, we get the equality
$$\sharp A = \Vol A - \sum_l (A(l)-\bar A(l))\cdot|(F^n=0)^l|,\eqno{(**)}$$where the sum is taken over all primitive $l\in M^*$, such that the affine span of $A^l$ contains 0. This is because all the other sets of the form $(\cdots)^l$ participating in the statement of Theorem 4.3, 
\cite{schon}, are empty by the considerations of dimension (in the sense of Definition \ref{dhomog}). Moreover, 
$(F^n=0)^l$ is empty unless the affine span of $A^l\cup\bar A^l$ is the whole $\R^n$.

In the latter case, the points in $(F^n=0)^l$ are counted by the subsequent Lemma \ref{ltech0}. Plugging its answer into $(**)$ and 

letting $p_B:\Z^n\to\Z^m$ be a projection whose kernel is generated by $B\subset\Z^n$, we get the first equality in the following chain (which proves the theorem):
$$\Vol A - \sharp A = \sum_l \Vol_\Z p_{A^l}(A^l\cup\bar A^l)\cdot\sum_{\Gamma\subset A^l}e_{A^l}^\Gamma\Vol_\Z\Gamma=$$ $$=\sum_{B\subset A}\Vol_\Z p_B((\conv A)\setminus\conv(A\setminus B))\cdot \sum_{\Gamma\subset B}e_{B}^\Gamma\Vol_\Z\Gamma=\sum_{\Gamma\subset A}e_{A}^\Gamma\Vol_\Z\Gamma.$$ Some explanations to this chain of equalities, proving Theorem \ref{thconcr2}:

1. The sum over $l$ is taken over all primitive $l$ such that the affine span of $A^l$ contains 0, and the other sums are taken over faces $\Gamma$ whose affine span contains 0;

2. The first equality is formula $(**)$ with $|(F^n=0)^l|$ plugged from Lemma \ref{ltech0};

3. The second equality reflects that $p_B((\conv A)\setminus\conv(A\setminus B))$ splits into pieces $p_{A^l}(A^l\cup\bar A^l)$ over all $l$ such that $B=A^l$, and thus its volume is the sum of the volumes of the pieces;

4. The third equality is the inductive definition of Euler obstruction $e^\Gamma_A$ (Proposition \ref{eulobstr}).

\subsection{The $l$-leading terms of $df$}

We study the $l$-leading terms of $df$ for the polynomial $f(x)=\sum_{a\in A}c_a x^a$ and a linear function $l:\Z^n\to\Z$ in case $A^l\cup\bar A^l$ is not contained in a hyperplane. More precisely, we are given a line $L$ not contained in $\{l=0\}$, and no face of $\conv(A^l\cup \bar A^l)$ has affine span containing the line $L\subset\Z^n$ (unless the span is the whole $\R^n$). We describe it with linear equations  $l_1=\cdots=l_{n-1}=0$, define polynomials $\partial_{l_i}  f$ as $\sum_{a\in A} l_i(a)c_ax^a$, and need to count the cardinality of the following set (in the notation of Theorem 4.3, 
\cite{schon}): $$(\partial_{l_1}  f=\cdots=\partial_{l_{n-1}} f=0)^l.\eqno{(\star)}$$Recall that this set does not depend on the choice of $l_i$'s (but only on the given line $L$), and is contained in the orbit $O$ of the toric variety $X_\Sigma$, corresponding to the cone of the fan $\Sigma$ which contains $l$ in its interior.
To formulate the answer, we denote by $p_B:\Z^n\to\Z^m$ a surjection whose kernel is generated by $B\subset\Z^n$, and by $k$ the dimension of the affine span of $A^l$. The answer will be proved assuming by induction that we have already proved Theorem \ref{thconcr2} for polynomials of less than $n$ variables.
\begin{lemma}\label{ltech0}
The number of points in the set $(\star)$ is $$\Vol_\Z p_{A^l+L}\bar A^l\cdot\sum_{\Gamma\subset A^l}e^\Gamma_{A^l}\Vol_\Z\Gamma,$$
where $\Gamma$ ranges over faces of $A^l$ whose affine spans contain 0 (including $A^l$ itself), $\Vol_\Z\Gamma$ is the lattice volume of $\Gamma$ in its affine span (which is always positive), and $\Vol_\Z p_{A^l+L}\bar A^l$ is the lattice volume of this set in $\Z^{n-k-1}$ (which is positive if $\dim p_{A^l+L}\bar A^l={n-k-1}$).
\end{lemma}

For the proof, we shall change the coordinates in the torus $\CC^n$ and its character lattice $\Z^n$, to bring the functions $l_1,\ldots,l_{n-1},l_n:=l$ and the line $L$ to the following convenient position.
\begin{assum}\label{assum1}
1. Of linear functions $l_1,\ldots,l_{n}$ on $\Z^n$, the first $l_1=\cdots=l_{n-1}=0$ define a line $L\subset\Z^n$, the last $l_{k+1}=\cdots=l_{n}=0$ define the plane generated by the first $k$ coordinate axes, and $l_n$ is the last coordinate function;

2. A finite subset $A\subset\Z^n$ is not contained in an affine hyperplane, and the two minimal values of $l_n$ on it equal $0$ and $m>0$;

3. For the subset $B:=A\cap\{l_n=0\}$, the affine span is the coordinate plane $l_{k+1}=\cdots=l_{n}=0$;

4. For the subset $\bar B:=A\cap\{l_n=m\}$, no face of $\conv(B\cup \bar B)$ has affine span containing the line $L\subset\Z^n$ (unless the span is the whole $\R^n$);

5. In the latter case, if the affine span of $B\cup \bar B$ is the whole $\R^n$, the line $L$ is not in the hyperplane  $l_n=0$ (or, equivalently, $l_1,\ldots,l_n$ are lineraly independent).
\end{assum}
In this setting, the Laurent polynomial $f\in\C^A$ is a regular function on $\C\times\CC^{n-1}\supset\CC^n$, the toric orbit $O$ containing the sought set $(\star)$ identifies with $H:=\{0\}\times\CC^{n-1}$, and the derivatives $\partial_{l_i} f,\,i>k$, are divisible by $x_n^m$. 

Thus 
the sought set $(\star)$ is defined by the following system of regular equations on $H$:
$$\partial_{l_1}f=\cdots=\partial_{l_k}f=x_n^{-m}\partial_{l_{k+1}}f=\cdots=x_n^{-m}\partial_{l_{n-1}}f=0.\eqno{(*)}$$Thus Lemma \ref{ltech0} takes the following form, denoting the coordinate projection $(l_{k+1},\ldots,l_{n-1}):\Z^n\to\Z^{n-k-1}$ by $p$.
\begin{lemma}\label{ltech1}
For generic $f\in\C^A$, the number of roots of system $(*)$ in $H$ equals $$\Vol_\Z p\bar B\cdot\sum_{\Gamma\subset B}e_B^\Gamma \Vol_\Z\Gamma.$$

\end{lemma}

\begin{proof}
Note that the subsystem $$\partial_{l_1}f=\cdots=\partial_{l_k}f=0$$on $H$ is equivalent to $df|_H=0$ and is square: it depends only on the first $k<n$ variables, because $f|_H\in\C^{B}$, and $B$ belongs to the first coordinate $k$-plane. Thus the number of roots $y\in\CC^k$ of this subsystem is given by Theorem \ref{thconcr2}: it is $\sum_{\Gamma\subset B}e_B^\Gamma\Vol_\Z\Gamma.$

Plugging such a root $y$ into $x_n^{-m}\partial_{l_n}f$ restricted to $H$, we get a generic polynomial $f_y\in\C^{p\bar B}$. (To put this formally, the multivalued map sending every $f\in\C^A$ to the collection of $f_y\in\C^{p\bar B}$ over all roots $y$ of system $(*)$, is dominant.) Thus the number of roots of the critical complete intersection $df_y=0$ on $\CC^{n-k-1}$ is again given by Theorem \ref{thconcr2}. And this time it equals $\Vol_\Z p\bar B$, because no proper face of $p\bar B$ has affine span containing 0 by Assumption \ref{assum1}.4.

Since the roots of the initial system $(*)$ are in one to one correspondence with the roots of the systems $df_y=0$ over all roots $y$ of the system $df|_H=0$, the total number of roots is the product of the numbers of roots of $df_y=0$ and $df|_H=0$. They are computed above.
\end{proof}

\vspace{1ex}
\noindent
London Institute for Mathematical Sciences, UK\\
\textit{Email}: aes@lims.ac.uk
\end{document}